\documentclass{scrartcl}

\usepackage{amsfonts}
\usepackage{amsmath}
\usepackage{amssymb}
\usepackage{hyperref}
\usepackage{pgf}

\newcommand{\bbbr}{\mathbb{R}}
\newcommand{\bbbn}{\mathbb{N}}

\newcommand{\Idx}{\mathcal{I}}
\newcommand{\Jdx}{\mathcal{J}}
\newcommand{\Kdx}{\mathcal{K}}
\newcommand{\ctI}{\mathcal{T}_{\Idx}}
\newcommand{\ctJ}{\mathcal{T}_{\Jdx}}
\newcommand{\ctK}{\mathcal{T}_{\Kdx}}
\newcommand{\lfI}{\mathcal{L}_{\Idx}}

\newcommand{\ctIJ}{\mathcal{T}_{\Idx\times\Jdx}}
\newcommand{\ctJK}{\mathcal{T}_{\Jdx\times\Kdx}}
\newcommand{\ctIK}{\mathcal{T}_{\Idx\times\Kdx}}
\newcommand{\lfIJ}{\mathcal{L}_{\Idx\times\Jdx}}
\newcommand{\lfJK}{\mathcal{L}_{\Jdx\times\Kdx}}

\newcommand{\lfaIJ}{\mathcal{L}_{\Idx\times\Jdx}^+}
\newcommand{\lfiIJ}{\mathcal{L}_{\Idx\times\Jdx}^-}
\newcommand{\lfaJK}{\mathcal{L}_{\Jdx\times\Kdx}^+}
\newcommand{\lfiJK}{\mathcal{L}_{\Jdx\times\Kdx}^-}
\newcommand{\lfaIK}{\mathcal{L}_{\Idx\times\Kdx}^+}

\newcommand{\ctIJK}{\mathcal{T}_{\Idx\times\Jdx\times\Kdx}}
\newcommand{\lfIJK}{\mathcal{L}_{\Idx\times\Jdx\times\Kdx}}
\newcommand{\lfaIJK}{\mathcal{L}_{\Idx\times\Jdx\times\Kdx}^+}

\newcommand{\Cp}{C_p}

\newcommand{\dist}{\mathop{\operatorname{dist}}\nolimits}
\newcommand{\diam}{\mathop{\operatorname{diam}}\nolimits}
\newcommand{\chil}{\mathop{\operatorname{chil}}\nolimits}
\newcommand{\desc}{\mathop{\operatorname{desc}}\nolimits}
\newcommand{\treeroot}{\mathop{\operatorname{root}}\nolimits}

\newcommand{\supp}{\mathop{\operatorname{supp}}\nolimits}

\newcommand{\accu}{\mathop{\operatorname{accu}}\nolimits}
\newcommand{\yield}{\mathop{\operatorname{yield}}\nolimits}
\newcommand{\level}{\mathop{\operatorname{level}}\nolimits}
\newcommand{\brow}{\mathop{\operatorname{row}}\nolimits}
\newcommand{\bcol}{\mathop{\operatorname{col}}\nolimits}

\newcommand{\trans}{\mathop{\operatorname{mul}}\nolimits}
\newcommand{\finish}{\mathop{\operatorname{finish}}\nolimits}
\newcommand{\add}{\mathop{\operatorname{add}}\nolimits}
\newcommand{\bsplit}{\mathop{\operatorname{split}}\nolimits}
\newcommand{\admbasistree}{\mathop{\operatorname{uniform}}\nolimits}
\newcommand{\inadmbasistree}{\mathop{\operatorname{nearfield}}\nolimits}

\newtheorem{theorem}{Theorem}
\newtheorem{definition}[theorem]{Definition}
\newtheorem{lemma}[theorem]{Lemma}

\newtheorem{remark}[theorem]{Remark}

\newenvironment{proof}{\emph{Proof.}}{\hfill $\Box$}

\allowdisplaybreaks

\begin{document}

\renewcommand{\topfraction}{0.9}
\renewcommand{\textfraction}{0.1}
\renewcommand{\floatpagefraction}{0.9}

\title{Adaptive multiplication of $\mathcal{H}^2$-matrices with
  block-relative error control}
\author{Steffen B\"orm}

\maketitle

\begin{abstract}
The discretization of non-local operators, e.g., solution operators
of partial differential equations or integral operators, leads to
large densely populated matrices.
$\mathcal{H}^2$-matrices take advantage of local low-rank structures
in these matrices to provide an efficient data-sparse approximation
that allows us to handle large matrices efficiently, e.g., to reduce
the storage requirements to $\mathcal{O}(n k)$ for $n$-dimensional
matrices with local rank $k$, and to reduce the complexity of the
matrix-vector multiplication to $\mathcal{O}(n k)$ operations.

In order to perform more advanced operations, e.g., to construct
efficient preconditioners or evaluate matrix functions, we require
algorithms that take $\mathcal{H}^2$-matrices as input and approximate
the result again by $\mathcal{H}^2$-matrices, ideally with controllable
accuracy.
In this manuscript, we introduce an algorithm that approximates
the product of two $\mathcal{H}^2$-matrices and guarantees block-relative
error estimates for the submatrices of the result.
It uses specialized tree structures to represent the exact product
in an intermediate step, thereby allowing us to apply mathematically
rigorous error control strategies.
\end{abstract}

\section{Introduction}

In order to treat partial differential or integral equations numerically,
discretization schemes are required that lead to systems of linear
equations that have to be solved in order to obtain approximations
of the solutions of the original equations.
Highly accurate approximations typically require large systems of
linear equations, and treating these large systems efficiently poses
a significant challenge.

In the context of elliptic partial differential equations, standard
finite element or finite difference discretization schemes lead to
\emph{sparse} matrices, i.e., to matrices with only a small number
of non-zero coefficients, and this structure can be used to handle
these matrices very efficiently.

The corresponding solution operators, however, are generally
\emph{dense} matrices, i.e., most or even all coefficients are
non-zero, and specialized techniques have to be used to treat these
matrices efficiently.

The same problem arises in the context of integral equations, where
standard discretization schemes immediately lead to dense matrices.

For integral equations, the \emph{fast multipole method}
\cite{RO85,GRRO87,GRRO97} offers an efficient approach: the kernel
function underlying the equation is locally approximated by
an expansion that leads to factorized low-rank representations
of submatrices and allows us to reduce the complexity for an
$n$-dimensional matrix to $\mathcal{O}(n k)$, where the local
rank $k$ can be used to control the accuracy of the approximation.

The closely related \emph{panel clustering method} \cite{HANO89}
extends these ideas to more general integral operators with a
complexity of $\mathcal{O}(n k \log n)$, while refined versions
reduce the complexity to $\mathcal{O}(n)$ by ensuring that the
matrix approximation is just accurate enough to keep track with
the discretization error \cite{SA00,BOSA03}.

\emph{Hierarchical matrices} \cite{HA99,GRHA02,HA15} use general
factorized low-rank representations to significantly extend the
scope of applications while keeping the close-to-optimal complexity
of $\mathcal{O}(n k \log n)$.
Using purely algebraic methods like the singular value decomposition
or rank-revealing factorizations allows hierarchical matrices to
approximate products and inverses of matrices, to obtain robust
preconditioners using an approximate LU factorization, and even
to evaluate matrix functions \cite{GAHAKH00} and solve certain matrix
equations \cite{GR01a,GRHAKH02,BA08} in $\mathcal{O}(n k^\alpha \log^\beta n)$
operations for small values of $\alpha,\beta>0$.

Most advanced applications of hierarchical matrices are based
on an algorithm for approximating the product of two
hierarchical matrices within a given block structure.
Applying this algorithm recursively to submatrices allows us
to approximate the inverse and the LU factorization, which
in turn give rise to algorithms for matrix functions.

\emph{$\mathcal{H}^2$-matrices} \cite{HAKHSA00,BOHA02,BO10} are
a special class of hierarchical matrices that makes use of nested
multi-level bases --- similar to fast multipole methods --- to reduce
the storage complexity to $\mathcal{O}(n k)$.
The nested representation of the bases appears as both a blessing
and a curse:
on the one hand, using nested bases makes algorithms considerably
more efficient, on the other hand, finding good nested bases for
the result of an algebraic operation poses a considerable challenge.

For very simple block structures \cite{HAKHKR04,CHGULY05,CHGULIXI09b},
efficient algorithms for performing important algebraic operations
with $\mathcal{H}^2$-matrices are available, but the more general
structures required for applications in two- and three-dimensional
space lead to complicated rank structures that are not easily taken
advantage of.
The central task is finding good bases for the result of an operation,
since once the bases are known, efficient orthogonal projections can
be applied to obtain the final result \cite{BO04a} in linear complexity
$\mathcal{O}(n k^2)$.

A first attempt for approximating the product of two
$\mathcal{H}^2$-matrices by an $\mathcal{H}^2$-matrix with
adaptively chosen bases guaranteeing a given accuracy uses a sequence
of local low-rank updates \cite{BORE14} with $\mathcal{O}(n k^2 \log n)$
complexity, but the relatively complicated memory access patterns of
these updates make the practical performance of this algorithm unattractive
for many applications.

In this article, a new algorithm is presented that prepares the
\emph{exact} representation of the product of two $\mathcal{H}^2$-matrices
in an intermediate step and then applies a rank-revealing factorization
to obtain an $\mathcal{H}^2$-matrix approximation of the product.
A major advantage of the new algorithm compared to its predecessors
is that it can guarantee rigorous relative error estimates for every
submatrix appearing in the result.

The next section of the article gives a short introduction to the
structure of $\mathcal{H}^2$-matrices and summarizes their most
important properties.
The following section presents an in-depth analysis of the structure
of the product of two $\mathcal{H}^2$-matrices and leads to a
representation of the \emph{exact} product using \emph{basis trees}.
Once we have the exact product at our disposal, the next section
describes how we can apply low-rank factorizations to construct an
approximation by an $\mathcal{H}^2$-matrix with adaptively chosen bases.
The final section contains numerical experiments demonstrating that
the new algorithm has a complexity of $\mathcal{O}(n k^2 \log n)$
and is indeed able to guarantee the given accuracy.

\section{\texorpdfstring{$\mathcal{H}^2$-matrices}{H2-matrices}}

$\mathcal{H}^2$-matrices can be considered the algebraic counterparts
of fast multipole approximations of non-local operators:
for a given matrix $G\in\bbbr^{\Idx\times\Jdx}$ with row index set $\Idx$
and column index set $\Jdx$, submatrices $G|_{\hat t\times\hat s}$
with $\hat t\subseteq\Idx$ and $\hat s\subseteq\Jdx$ are chosen that
have numerically low rank and can therefore be approximated by
low-rank factorizations.

We consider a matrix resulting from the Galerkin discretization of an
integral operator as a motivating example:
let $(\varphi_i)_{i\in\Idx}$ and $(\psi_j)_{j\in\Jdx}$ be families of
basis functions on a domain or manifold $\Omega$, and let
\begin{equation*}
  g\colon \Omega\times\Omega\to\bbbr
\end{equation*}
be a kernel function.
We consider the matrix $G\in\bbbr^{\Idx\times\Jdx}$ given by
\begin{align}\label{eq:matrix}
  g_{ij} &= \int_\Omega \varphi_i(x) \int_\Omega g(x,y) \psi_j(y) \,dy \,dx &
  &\text{ for all } i\in\Idx,\ j\in\Jdx.
\end{align}
In many important applications, the kernel function $g$ is smooth
outside the diagonal, i.e., $g|_{\Omega_t\times\Omega_s}$ can be
efficiently approximated by polynomials if the \emph{target and source
subsets} $\Omega_t,\Omega_s\subseteq\Omega$ are well-separated
\cite[Chapter~4]{BO10}, i.e., if the
\emph{admissibility condition}
\begin{equation}\label{eq:admissibility}
  \max\{\diam(\Omega_t),\diam(\Omega_s)\}
  \leq 2 \eta \dist(\Omega_t,\Omega_s)
\end{equation}
holds with a parameter $\eta>0$.
Fixing interpolation points $(\xi_{t,\nu})_{\nu=1}^k$ and
$(\xi_{s,\mu})_{\mu=1}^k$ for $\Omega_t$ and $\Omega_s$, the interpolation
of the kernel function yields
\begin{align}\label{eq:interpolation}
  g(x,y) &\approx \sum_{\nu=1}^k \sum_{\mu=1}^k
              g(\xi_{t,\nu}, \xi_{s,\mu})
              \ell_{t,\nu}(x) \ell_{s,\mu}(y) &
  &\text{ for all } x\in\Omega_t,\ y\in\Omega_s,
\end{align}
where $(\ell_{t,\nu})_{\nu=1}^k$ and $(\ell_{s,\mu})_{\mu=1}^k$ are the
Lagrange polynomials for $\Omega_t$ and $\Omega_s$.
For the matrix entries, this approximation yields
\begin{align*}
  g_{ij} &\approx \int_\Omega \varphi_i(x) \int_\Omega
             \sum_{\nu=1}^k \sum_{\mu=1}^k
             g(\xi_{t,\nu}, \xi_{s,\mu})
             \ell_{t,\nu}(x) \ell_{s,\mu}(y) \psi_j(y) \,dy \,dx\\
  &= \sum_{\nu=1}^k \sum_{\mu=1}^k
     \underbrace{\int_\Omega \varphi_i(x) \ell_{t,\nu}(x) \,dx}_{=:v_{t,i\nu}}
     \underbrace{g(\xi_{t,\nu},\xi_{s,\mu})}_{=:s_{ts,\nu\mu}}
     \underbrace{\int_\Omega \psi_j(y) \ell_{s,\mu}(y) \,dy}_{=:w_{s,j\mu}}
\end{align*}
for all $i\in\Idx$ with $\supp\varphi_i\subseteq\Omega_t$ and all
$j\in\Jdx$ with $\supp\psi_j\subseteq\Omega_s$.
Introducing the index subsets
\begin{align*}
  \hat t &:= \{ i\in\Idx\ :\ \supp\varphi_i\subseteq\Omega_t \}, &
  \hat s &:= \{ j\in\Jdx\ :\ \supp\psi_j\subseteq\Omega_s \}
\end{align*}
and collecting the entries $v_{t,i\nu}$, $w_{s,j\mu}$ and $s_{ts,\nu\mu}$
in matrices $V_t\in\bbbr^{\hat t\times k}$, $W_s\in\bbbr^{\hat s\times k}$,
and $S_{ts}\in\bbbr^{k\times k}$, we can write this equation in
the compact form
\begin{equation}\label{eq:vsw}
  G|_{\hat t\times\hat s} \approx V_t S_{ts} W_s^*,
\end{equation}
where $W_s^*$ denotes the transposed of the matrix $W_s$.
The admissibility condition (\ref{eq:admissibility}) guarantees that
interpolation converges rapidly, so the number $k$ of interpolation
points will typically be significantly smaller than $|\hat t|$
and $|\hat s|$, which makes (\ref{eq:vsw}) a low-rank approximation
of the submatrix $G|_{\hat t\times\hat s}$.

Our task is to split the matrix $G$ into submatrices $G|_{\hat t\times\hat s}$
that can either be approximated, i.e., that satisfy the admissibility
condition (\ref{eq:admissibility}), or that are small enough to be
stored explicitly.
\emph{Cluster trees} have proven to be a very useful tool for constructing
these submatrices efficiently.
In the following definition, we focus on the index sets $\hat t$ and
assume that the corresponding subdomains $\Omega_t$ with
\begin{align*}
  \supp\varphi_i &\subseteq \Omega_t &
  &\text{ for all } i\in\hat t
\end{align*}
are provided by a suitable algorithm.

%
%
\begin{definition}[Cluster tree]
\label{de:cluster_tree}
Let $\mathcal{T}$ be a finite labeled tree, where the label of
each node $t\in\mathcal{T}$ is denoted by $\hat t$ and the set of children
is denoted by $\chil(\mathcal{T},t)$ (or short $\chil(t)$ if the tree
is clear from the context).
$\mathcal{T}$ is called a \emph{cluster tree} for the index set $\Idx$ if
\begin{itemize}
  \item for every node $t\in\mathcal{T}$, the label $\hat t$ is a subset
     of the index set $\Idx$,
  \item the label of the root is $\Idx$,
  \item for $t\in\mathcal{T}$ and $t_1,t_2\in\chil(t)$ we have
    $t_1\neq t_2 \Rightarrow \hat t_1\cap\hat t_2=\emptyset$, i.e.,
    the labels of siblings are disjoint, and
  \item for $t\in\mathcal{T}$ with $\chil(t)\neq\emptyset$ we have
    $\hat t = \bigcup_{t'\in\chil(t)} \hat t'$, i.e., the labels of
    the children correspond to a disjoint partition of the label
    of their parent.
\end{itemize}
A cluster tree for the index set $\Idx$ is usually denoted by $\ctI$,
its nodes are called \emph{clusters}, and its leaves are denoted by
$\lfI := \{ t\in\ctI\ :\ \chil(t)=\emptyset \}$.
\end{definition}

Sometimes we have to work with all clusters descended from a given
cluster $t$.
This corresponding \emph{set of descendants} is defined inductively
by
\begin{align*}
  \desc(t) &:= \{t\} \cup \bigcup_{t'\in\chil(t)} \desc(t') &
  &\text{ for all } t\in\ctI.
\end{align*}

A simple induction shows that the labels of the leaf clusters $\lfI$
form a disjoint partition of the index set $\Idx$.

Using two cluster trees $\ctI$ and $\ctJ$ for the row and column
index sets $\Idx$ and $\Jdx$, we can construct a partition of the
index set $\Idx\times\Jdx$ of a matrix $G\in\bbbr^{\Idx\times\Jdx}$
that can serve as a foundation for the approximation process.

%
%
\begin{definition}[Block tree]
\label{de:block_tree}
Let $\ctI$ and $\ctJ$ be cluster trees for index sets $\Idx$ and $\Jdx$.
Let $\mathcal{T}$ be a finite labeled tree, where the label of a node
$b\in\mathcal{T}$ is denoted by $\hat b$ and its children by
$\chil(b)$.

$\mathcal{T}$ is called a \emph{block tree} for $\ctI$ and $\ctJ$ if
\begin{itemize}
  \item for every node $b\in\mathcal{T}$, there are $t\in\ctI$ and
     $s\in\ctJ$ with $b=(t,s)$ and $\hat b=\hat t\times\hat s$,
  \item the root $b=\treeroot(\mathcal{T})$ of $\mathcal{T}$ is the pair
     of the roots of $\ctI$ and
     $\ctJ$, i.e., $b=(\treeroot(\ctI),\treeroot(\ctJ))$,
  \item for every node $b=(t,s)\in\mathcal{T}$ with $\chil(b)\neq\emptyset$
     we have
     \begin{equation}\label{eq:block_children}
       \chil(b) = \begin{cases}
         \{t\}\times\chil(s) &\text{ if } \chil(t)=\emptyset,\\
         \chil(t)\times\{s\} &\text{ if } \chil(s)=\emptyset,\\
         \chil(t)\times\chil(s) &\text{ otherwise}.
       \end{cases}
     \end{equation}
\end{itemize}
A block tree for $\ctI$ and $\ctJ$ is usually denoted by $\ctIJ$, its
nodes a called \emph{blocks}, and its leaves are denoted by
$\lfIJ := \{ b\in\ctIJ\ :\ \chil(b)=\emptyset \}$.

For $b=(t,s)\in\ctIJ$, $t$ and $s$ are called the \emph{row}
and \emph{column cluster} of $b$, respectively.
\end{definition}

If $\ctI$ and $\ctJ$ are cluster trees for the index sets $\Idx$
and $\Jdx$, a block tree $\ctIJ$ is a special cluster tree for the
Cartesian product $\Idx\times\Jdx$.
This implies that the leaves of the block tree form a disjoint
partition of $\Idx\times\Jdx$, i.e., they describe a decomposition
of a matrix $G\in\bbbr^{\Idx\times\Jdx}$ into non-overlapping
submatrices $G|_{\hat t\times\hat s}$ for $b=(t,s)\in\lfIJ$.

Our goal is now to split a given matrix $G\in\bbbr^{\Idx\times\Jdx}$
into submatrices that either are admissible, i.e., can be approximated
by low rank, or small, so that we can afford to store them directly.
In the interest of efficiency, we want the number of submatrices to
be as small as possible.

This task can be solved by a simple recursive procedure:
we start with the root $(\treeroot(\ctI),\treeroot(\ctJ))$ of the
block tree $\ctIJ$.

Given a block $b=(t,s)\in\ctIJ$, we check if it is admissible, i.e.,
if (\ref{eq:admissibility}) holds.
If $b$ is admissible, we keep it as a leaf of the block tree.

If $b$ is not admissible, we check whether $t$ and $s$ have children.
If at least one of them does, we construct the children of $b$ via
(\ref{eq:block_children}) and proceed to check them recursively.

If $t$ and $s$ both have no children, we keep $b$ as an
inadmissible leaf of the block tree.

This procedure yields a block tree with leaves that are either
admissible or consist of clusters that have no children, i.e., the
leaves $\lfIJ$ of the block tree $\ctIJ$ are split into the admissible
leaves $\lfaIJ$ satisfying (\ref{eq:admissibility}) and
the inadmissible leaves $\lfiIJ$.

For inadmissible leaves $b=(t,s)\in\lfiIJ$, our construction implies
that $t$ and $s$ are leaves of the cluster trees $\ctI$ and $\ctJ$, and
we can ensure that leaf clusters have only a small number of indices, so that
$G|_{\hat t\times\hat s}$ is a small matrix that we can afford to store
explicitly.

For admissible leaves $b=(t,s)\in\lfaIJ$, we can use the low-rank
approximation (\ref{eq:vsw}).
This approximation has the useful property that $V_t$ and $W_s$ depend
only on the row cluster $t$ and the column cluster $s$, respectively,
while only the small $k\times k$ matrix $S_{ts}$ depends on both clusters.
This means that we only have to store the \emph{coupling matrix}
$S_{ts}$ for every admissible leaf $b=(t,s)\in\lfaIJ$.

The matrix families $(V_t)_{t\in\ctI}$ and $(W_s)_{s\in\ctJ}$ have a
property that we can use to improve the efficiency of our approximation
even further:
let $t\in\ctI$ and $t'\in\chil(t)$.
We can assume $\Omega_{t'}\subseteq\Omega_t$, and if we use the same
order of interpolation for all clusters, the identity theorem for
polynomials yields
\begin{align*}
  \ell_{t,\nu}
  &= \sum_{\nu'=1}^k
    \underbrace{\ell_{t,\nu}(\xi_{t',\nu'})}_{=:e_{t',\nu'\nu}}
    \ell_{t',\nu'} &
  &\text{ for all } \nu\in[1:k].
\end{align*}
For every $i\in\hat t'$, we have $i\in\hat t$ and therefore
\begin{align*}
  v_{t,i\nu}
  &= \int_\Omega \varphi_i(x) \ell_{t,\nu}(x) \,dx
   = \sum_{\nu'=1}^k \int_\Omega \varphi_i(x) \ell_{t',\nu'}(x) \,dx
              \,e_{t',\nu'\nu}
   = \sum_{\nu'=1}^k v_{t',i\nu'} e_{t',\nu'\nu}.
\end{align*}
In short, we have $V_t|_{\hat t'\times k} = V_{t'} E_{t'}$.
By Definition~\ref{de:cluster_tree}, every index $i\in\hat t$ appears
in exactly one child $t'\in\chil(t)$, therefore $V_t$ is completely
determined by the matrices $V_{t'}$ for the children and the
\emph{transfer matrices} $E_{t'}$.

%
%
\begin{definition}[Cluster basis]
\label{de:cluster_basis}
Let $V=(V_t)_{t\in\ctI}$ be a family of matrices satisfying
$V_t\in\bbbr^{\hat t\times k}$ for every $t\in\ctI$.
$V$ is called a \emph{cluster basis} for $\ctI$ if for every $t\in\ctI$ with
$t'\in\chil(t)$ there is a matrix $E_{t'}\in\bbbr^{k\times k}$
such that
\begin{equation}\label{eq:transfer}
  V_t|_{\hat t'\times k} = V_{t'} E_{t'}.
\end{equation}
The matrices $E_{t'}$ are called \emph{transfer matrices}.
The matrices $V_t$ for leaf clusters $t\in\lfI$ are called \emph{leaf matrices}.
\end{definition}

A simple induction yields that the entire cluster basis $V$ is defined
by the transfer matrices $E_t$ for all $t\in\ctI$ and the leaf matrices
$V_t$ for leaf clusters $t\in\lfI$.
It is important to note that the matrices $V_t$ for non-leaf clusters
$t\in\ctI\setminus\lfI$ do not have to be stored, but can be reconstructed
using the transfer matrices.
Under standard assumptions, this reduces the storage requirements for
the cluster basis from $\mathcal{O}(n k \log n)$ to $\mathcal{O}(n k)$
\cite{BOHA02,BO10}.

While the equation (\ref{eq:transfer}) holds exactly for kernel expansions
with a constant polynomial order, it can be replaced by analytic
\cite{BO22} or algebraic approximations \cite{BOHA02} while keeping
the error under control.

%
%
\begin{definition}[$\mathcal{H}^2$-matrix]
Let $V=(V_t)_{t\in\ctI}$ and $W=(W_s)_{s\in\ctJ}$ be cluster bases
for $\ctI$ and $\ctJ$, and let $\ctIJ$ be a block tree for $\ctI$ and $\ctJ$.
A matrix $G\in\bbbr^{\Idx\times\Jdx}$ is called an
\emph{$\mathcal{H}^2$-matrix} with the \emph{row basis} $V$ and
the \emph{column basis} $W$ if for every admissible leaf
$b=(t,s)\in\lfaIJ$ there is a matrix $S_{ts}\in\bbbr^{k\times k}$ with
\begin{equation}\label{eq:vsw_exact}
  G|_{\hat t\times\hat s} = V_t S_{ts} W_s^*.
\end{equation}
The matrix $S_{ts}$ is called the \emph{coupling matrix} for the
block $b$.
\end{definition}

Under standard assumptions, an $\mathcal{H}^2$-matrix can be represented
by the coupling matrices $S_{ts}$ for all admissible leaves $b=(t,s)\in\lfaIJ$
and the \emph{nearfield matrices} $G|_{\hat t\times\hat s}$ for all
inadmissible leaves $b=(t,s)\in\lfiIJ$ in $\mathcal{O}(n k)$ units of
storage \cite{BOHA02,BO10}.
For $(t,s)\in\ctIJ$, the matrix-vector multiplications with submatrices
$x \mapsto G|_{\hat t\times\hat s} x$ and
$x \mapsto G|_{\hat t\times\hat s}^* x$ require
$\mathcal{O}((|\hat t|+|\hat s|) k)$ operations.
The multiplication with an $n$-dimensional $\mathcal{H}^2$-matrix
requires $\mathcal{O}(n k)$ operations.

\section{Products of \texorpdfstring{$\mathcal{H}^2$-matrices}{H2-matrices}}

While the multiplication of an $\mathcal{H}^2$-matrix with a vector
is a relatively simple operation, multiplying two $\mathcal{H}^2$-matrices
poses a significant challenge, particularly if we want to approximate
the result again by an efficient $\mathcal{H}^2$-matrix representation.

In this section, we are looking for structures in the product of
two $\mathcal{H}^2$-matrices that we can use to construct an efficient
algorithm.

%
%
\subsection{Semi-uniform matrices}

Let $\ctI$, $\ctJ$ and $\ctK$ be cluster trees for the index
sets $\Idx$, $\Jdx$ and $\Kdx$.
Let $X\in\bbbr^{\Idx\times\Jdx}$ be an $\mathcal{H}^2$-matrix with
row basis $V_X=(V_{X,t})_{t\in\ctI}$ and column basis
$W_X=(W_{X,s})_{s\in\ctJ}$ for the block tree $\ctIJ$,
and let $Y\in\bbbr^{\Jdx\times\Kdx}$ be an $\mathcal{H}^2$-matrix with
row basis $V_Y=(V_{Y,s})_{s\in\ctJ}$ and column basis
$W_Y=(W_{Y,r})_{r\in\ctK}$ for the block tree $\ctJK$.
We denote the transfer matrices for $V_X$, $W_X$, $V_Y$, and $W_Y$
by $(E_{X,t})_{t\in\ctI}$, $(F_{X,s})_{s\in\ctJ}$, $(E_{Y,s})_{s\in\ctJ}$,
and $(F_{Y,r})_{r\in\ctK}$.

Our goal is to approximate the product $Z = X Y$ with a given error
tolerance.
We approach this task by deriving a representation for the \emph{exact}
product and then applying a compression algorithm to obtain a
sufficiently accurate approximation.

In order to be able to proceed by recursion, we consider sub-products
\begin{equation}\label{eq:subproduct}
  X|_{\hat t\times\hat s} Y|_{\hat s\times\hat r}
\end{equation}
with $(t,s)\in\ctIJ$ and $(s,r)\in\ctJK$.
If $(t,s)$ and $(s,r)$ both have children in $\ctIJ$ and $\ctJK$,
respectively, the sub-product (\ref{eq:subproduct}) can be expressed
by switching recursively to the children.
Definition~\ref{de:block_tree} ensures that products of the children
will again be of the form (\ref{eq:subproduct}).

A far more interesting case appears if $(t,s)$ or $(s,r)$ are leaves
of their respective block trees.
If $(s,r)\in\lfaJK$ holds, we have $Y|_{\hat s\times\hat r}
= V_{Y,s} S_{Y,sr} W_{Y,r}^*$ due to (\ref{eq:vsw_exact}) and
therefore
\begin{subequations}\label{eq:semiuniform}
\begin{equation}\label{eq:right_semiuniform}
  X|_{\hat t\times\hat s} Y|_{\hat s\times\hat r}
  = X|_{\hat t\times\hat s} V_{Y,s} S_{Y,sr} W_{Y,r}^*
  = A_{tr} W_{Y,r}^*
\end{equation}
with the matrix $A_{tr} := X|_{\hat t\times\hat s} V_{Y,s}
S_{Y,sr}\in\bbbr^{\hat t\times k}$.

If $(t,s)\in\lfaIJ$ holds, we have $X|_{\hat t\times\hat s}
= V_{X,t} S_{X,ts} W_{X,s}^*$ due to (\ref{eq:vsw_exact}) and
therefore
\begin{equation}\label{eq:left_semiuniform}
  X|_{\hat t\times\hat s} Y|_{\hat s\times\hat r}
  = V_{X,t} S_{X,ts} W_{X,s}^* Y|_{\hat s\times\hat r}
  = V_{X,t} B_{tr}^*
\end{equation}
\end{subequations}
with the matrix $B_{tr} := Y|_{\hat s\times\hat r}^* W_{X,s} S_{X,ts}^*
\in\bbbr^{\hat r\times k}$.

We call matrices of the form (\ref{eq:right_semiuniform})
\emph{right semi-uniform} and matrices of the form
(\ref{eq:left_semiuniform}) \emph{left semi-uniform}.
Sums of both, i.e., matrices of the form
\begin{equation*}
  A_{tr} W_{Y,r}^* + V_{X,t} B_{tr}^*,
\end{equation*}
are called \emph{semi-uniform}.
Semi-uniform matrices are a subspace of $\bbbr^{\hat t\times\hat r}$,
since for two semi-uniform matrices $A_{tr} W_{Y,r}^* + V_{X,t} B_{tr}^*$
and $C_{tr} W_{Y,r}^* + V_{X,t} D_{tr}^*$ we have
\begin{equation*}
  (A_{tr} W_{Y,r}^* + V_{X,t} B_{tr}^*)
  + (C_{tr} W_{Y,r}^* + V_{X,t} D_{tr}^*)
  = (A_{tr} + C_{tr}) W_{Y,r}^* + V_{X,t} (B_{tr} + D_{tr})^*,
\end{equation*}
i.e., the sum can be computed by adding $A_{tr}$ and $C_{tr}$ as
well as $B_{tr}$ and $D_{tr}$, respectively.

Using semi-uniform matrices, we can express submatrices
$Z|_{\hat t\times\hat r}$ of the result $Z=XY$ efficiently and
exactly.

%
%
\subsection{Accumulators}

The special structure of semi-uniform matrices allows us to
represent submatrices of the product efficiently in a suitable
data structure.

%
%
\begin{definition}[Accumulator]
\label{de:accumulator}
Let $t\in\ctI$ and $r\in\ctK$.

The quadruple $\mathcal{A}_{tr} := (A_{tr}, B_{tr}, N_{tr}, \mathcal{P}_{tr})$
with
\begin{itemize}
  \item $A_{tr}\in\bbbr^{\hat t\times k}$ for right-semiuniform matrices,
  \item $B_{tr}\in\bbbr^{\hat r\times k}$ for left-semiuniform matrices,
  \item $N_{tr}\in\bbbr^{\hat t\times\hat r}$ for inadmissible leaves and
  \item $\mathcal{P}_{tr} \subseteq \{ (s,X,Y)
      \ :\ s\in\ctJ,\ X\in\bbbr^{\hat t\times\hat s},
         \ Y\in\bbbr^{\hat s\times\hat r} \}$ for subdivided products
\end{itemize}
is called an \emph{accumulator} for products $X|_{\hat t\times\hat s}
Y|_{\hat s\times\hat r}$.
It represents the matrix
\begin{equation*}
  \accu(\mathcal{A}_{tr})
  := A_{tr} W_{Y,r}^* + V_{X,t} B_{tr}^* + N_{tr}
     + \sum_{(s,X,Y)\in\mathcal{P}_{tr}} X Y.
\end{equation*}
\end{definition}

An accumulator can be used to represent a sum of products
$X|_{\hat t\times\hat s} Y|_{\hat s\times\hat r}$:
we start with an accumulator representing the zero matrix, i.e.,
with $A_{tr}=0$, $B_{tr}=0$, $N_{tr}=0$, and $\mathcal{P}_{tr}=\emptyset$.
If we add a product $X|_{\hat t\times\hat s} Y|_{\hat s\times\hat r}$ to
the accumulator, we distinguish four cases:
\begin{enumerate}
  \item If $(s,r)\in\lfaJK$, we add $X|_{\hat t\times\hat s} V_{Y,s} S_{Y,sr}$
     to $A_{tr}$, which corresponds to adding
     \begin{equation*}
       X|_{\hat t\times\hat s} V_{Y,s} S_{Y,sr} W_{Y,r}^*
       = X|_{\hat t\times\hat s} Y|_{\hat s\times\hat r}
     \end{equation*}
     to $\accu(\mathcal{A}_{tr})$.
  \item If $(t,s)\in\lfaIJ$, we add $Y|_{\hat s\times\hat r}^* W_{X,s}
     S_{X,ts}^*$ to $B_{tr}$, which corresponds to adding
     \begin{equation*}
       V_{X,t} (Y|_{\hat s\times\hat r}^* W_{X,s} S_{X,ts}^*)^*
       = V_{X,t} S_{X,ts} W_{X,s}^* Y|_{\hat s\times\hat r}
       = X|_{\hat t\times\hat s} Y|_{\hat s\times\hat r}
     \end{equation*}
     to $\accu(\mathcal{A}_{tr})$.
  \item If $(t,s)\in\lfiIJ$ or $(s,r)\in\lfiJK$, we add the
     product $X|_{\hat t\times\hat s} Y|_{\hat s\times\hat r}$ to
     $N_{tr}$, and therefore also to $\accu(\mathcal{A}_{tr})$.
  \item Otherwise, we add the triple
     $(s,X|_{\hat t\times\hat s},Y|_{\hat s\times\hat r})$
     to the set $\mathcal{P}_{tr}$, which again corresponds to
     adding $X|_{\hat t\times\hat s} Y|_{\hat s\times\hat r}$
     to $\accu(\mathcal{A}_{tr})$.
\end{enumerate}
The first two cases require $k$ $\mathcal{H}^2$-matrix-vector
multiplications and therefore have a complexity of
$\mathcal{O}((|\hat t|+|\hat s|) k^2)$ and
$\mathcal{O}((|\hat s|+|\hat r|) k^2)$, respectively.
In the third case, either $X|_{\hat t\times\hat s}$ or
$Y|_{\hat s\times\hat r}$ is a small matrix, allowing us to update
$N_{tr}$ efficiently again.
The fourth case does not involve any computation, merely book-keeping.

%
%
\subsection{Splitting accumulators}

The set $\mathcal{P}_{tr}$ collects products that may have a complicated
structure and are therefore not immediately accessible for efficient
computation.
Our goal is to reach an accumulator with $\mathcal{P}_{tr}=\emptyset$,
since then we have an efficient low-rank representation of the result.

To achieve this goal, we ``split'' the accumulator, i.e., for all
children $t'\in\chil(t)$ and $r'\in\chil(r)$ we construct
accumulators $\mathcal{A}_{t'r'}$ with
\begin{equation*}
  \accu(\mathcal{A}_{t'r'}) = \accu(\mathcal{A}_{tr})|_{\hat t'\times\hat r'},
\end{equation*}
i.e., the accumulator $\mathcal{A}_{t'r'}$ represents a submatrix
of the matrix represented by $\mathcal{A}_{tr}$.
Splitting the accumulator is achieved by handling its components
individually:
\begin{itemize}
  \item For $A_{tr}$, we take advantage of the nested cluster basis
    (\ref{eq:transfer}),
    i.e., $W_{Y,r}|_{\hat r'\times k} = W_{Y,r'} F_{Y,r'}$ with the
    transfer matrix $F_{Y,r'}$, to find
    \begin{align*}
      (A_{tr} W_{Y,r}^*)|_{\hat t'\times\hat r'}
      &= A_{tr}|_{\hat t'\times k} (W_{Y,r}|_{\hat r'\times k})^*\\
      &= A_{tr}|_{\hat t'\times k} (W_{Y,r'} F_{Y,r'})^*
       = A_{tr}|_{\hat t'\times k} F_{Y,r'}^* W_{Y,r'}^*,
    \end{align*}
    so we choose $A_{t'r'} := A_{tr}|_{\hat t'\times k} F_{Y,r'}^*$.
  \item For $B_{tr}$, we again use (\ref{eq:transfer}), i.e.,
    $V_{X,t}|_{\hat t'\times k} = V_{X,t'} E_{X,t'}$ with the
    transfer matrix $E_{X,t'}$, to find
    \begin{align*}
      (V_{X,t} B_{tr}^*)|_{\hat t'\times\hat r'}
      &= V_{X,t}|_{\hat t'\times k} (B_{tr}|_{\hat r'\times k})^*\\
      &= V_{X,t'} E_{X,t'} (B_{tr}|_{\hat r'\times k})^*
       = V_{X,t'} (B_{tr}|_{\hat r'\times k} E_{X,t'}^*)^*,
    \end{align*}
    so we choose $B_{t'r'} := B_{tr}|_{\hat r'\times k} E_{X,t'}^*$.
  \item For $N_{tr}$, we can simply use $N_{t'r'} := N_{tr}|_{\hat
      t'\times\hat r'}$.
  \item For $\mathcal{P}_{tr}$, we take all triples
     $(s,X,Y)\in\mathcal{P}_{tr}$ and add the sub-products
     $X|_{\hat t'\times\hat s'} Y|_{\hat s'\times\hat r'}$
     for $s'\in\chil(s)$ to the accumulator $\mathcal{A}_{t'r'}$ using the
     procedure described above.
\end{itemize}
This algorithm provides us with accumulators $\mathcal{A}_{t'r'}$
that represents the submatrices
$\accu(\mathcal{A}_{tr})|_{\hat t'\times\hat r'}$ for all $t'\in\chil(t)$
and $r'\in\chil(r)$ exactly.

%
%
\subsection{Induced block structure}

The condition $\mathcal{P}_{tr}=\emptyset$ required to obtain a
practically useful representation of the accumulator cannot be
guaranteed for \emph{all} pairs $(t,r)$ of clusters, but we can
limit the number of pairs violating this condition.

%
%
\begin{lemma}[Product admissibility]
\label{le:product_admissibility}
Let $t\in\ctI$ and $r\in\ctK$.
If $\mathcal{P}_{tr}$ is not empty, we have
\begin{equation*}
  \frac{\eta}{\eta+1} \dist(\Omega_t,\Omega_r)
  < \max\{\diam(\Omega_t),\diam(\Omega_s),\diam(\Omega_r)\}
\end{equation*}
for a cluster $s\in\ctJ$ with $(t,s)\in\ctIJ\setminus\lfIJ$ and
$(s,r)\in\ctJK\setminus\lfJK$.
\end{lemma}
\begin{proof}
Assume that $\mathcal{P}_{tr}$ is not empty.
We choose a triple $(s,X,Y)\in\mathcal{P}_{tr}$.
By our construction, it can only have been added to $\mathcal{P}_{tr}$
if both $(t,s)$ and $(s,r)$ are non-leaf blocks in $\ctIJ$ and $\ctJK$,
i.e., if $(t,s)\in\ctIJ\setminus\lfIJ$ and $(s,r)\in\ctJK\setminus\lfJK$.
This means that $(t,s)$ and $(s,r)$ are not admissible, i.e.,
\begin{align*}
  \max\{\diam(\Omega_t),\diam(\Omega_s)\}
  &> 2 \eta \dist(\Omega_t,\Omega_s),\\
  \max\{\diam(\Omega_s),\diam(\Omega_r)\}
  &> 2 \eta \dist(\Omega_s,\Omega_r).
\end{align*}
Using the triangle inequality, we find
\begin{align*}
  \dist(\Omega_t,\Omega_r)
  &\leq \dist(\Omega_t,\Omega_s) + \diam(\Omega_s)
        + \dist(\Omega_s,\Omega_r)\\
  &< \frac{1}{2\eta} \max\{\diam(\Omega_t),\diam(\Omega_s)\}
     + \diam(\Omega_s)\\
  &\qquad + \frac{1}{2\eta} \max\{\diam(\Omega_s),\diam(\Omega_r)\}\\
  &\leq \left(1 + \frac{1}{\eta}\right)
        \max\{\diam(\Omega_t),\diam(\Omega_s),\diam(\Omega_r)\}.
\end{align*}
This is the inequality we have claimed.
\end{proof}

During the construction of the cluster trees and block trees, we
can usually ensure that $(t,s)\in\ctIJ\setminus\lfIJ$ implies
that $\diam(\Omega_t)$ and $\diam(\Omega_s)$ are of a similar size,
so that the condition of Lemma~\ref{le:product_admissibility}
is reduced to
\begin{equation*}
  \frac{\eta}{\eta+1} \dist(\Omega_t,\Omega_r)
  \lesssim \max\{\diam(\Omega_t),\diam(\Omega_r)\},
\end{equation*}
i.e., the set $\mathcal{P}_{tr}$ can only be non-empty if the
clusters $t$ and $r$ are relatively close to each other.
Comparing this condition to the admissibility condition
(\ref{eq:admissibility}) indicates that, for sufficiently
adjusted parameters, admissible pairs $(t,r)$ will correspond
to efficiently usable accumulators with $\mathcal{P}_{tr}=\emptyset$.

For cluster trees constructed by the usual algorithms, only
a limited number of clusters can be this close to each other,
therefore the cardinalities of the sets
\begin{align*}
  \mathcal{I}_t &:= \{ r\in\ctK\ :\ \mathcal{P}_{tr}\neq\emptyset,
      \ \level(t)=\level(r) \}
  &\text{ for all } t\in\ctI,\\
  \mathcal{I}_r &:= \{ t\in\ctI\ :\ \mathcal{P}_{tr}\neq\emptyset,
      \ \level(t)=\level(r) \}
  &\text{ for all } r\in\ctK
\end{align*}
are bounded by a constant $\Cp\in\bbbn$.

In order to keep track of all sub-products required to compute
the product of two $\mathcal{H}^2$-matrices $X\in\bbbr^{\Idx\times\Jdx}$
and $Y\in\bbbr^{\Jdx\times\Kdx}$ corresponding to block trees
$\ctIJ$ and $\ctJK$, it is useful to organize the products in
another tree structure.

%
%
\begin{definition}[Product tree]
Let $\ctIJ$ and $\ctJK$ be block trees for the cluster trees
$\ctI$, $\ctJ$ and $\ctJ$, $\ctK$, respectively.
Let $\mathcal{T}$ be a finite tree, where the children of a
node $p\in\mathcal{T}$ are denoted by $\chil(p)$.

$\mathcal{T}$ is called a \emph{product tree} for $\ctIJ$ and $\ctJK$ if
\begin{itemize}
  \item for every node $p\in\mathcal{T}$ there are $t\in\ctI$, $s\in\ctJ$
      and $r\in\ctK$ with $p=(t,s,r)$,
  \item the root $p=\treeroot(\mathcal{T})$ is the tuple of the roots
      of $\ctI$, $\ctJ$ and $\ctK$, i.e.,
      $p=(\treeroot(\ctI),\treeroot(\ctJ),\treeroot(\ctK))$,
  \item for every node $p=(t,s,r)\in\mathcal{T}$ we have
    \begin{equation*}
      \chil(p) = \{ (t',s',r')\in\ctI\times\ctJ\times\ctK
          \ :\ (t',s')\in\chil(t,s),\ (s',r')\in\chil(s,r) \}.
    \end{equation*}
\end{itemize}
The product tree for $\ctIJ$ and $\ctJK$ is usually denoted by
$\ctIJK$, and its leaves are denoted by $\lfIJK := \{ p\in\ctIJK
\ :\ \chil(p)=\emptyset \}$.
\end{definition}

Due to the definition, $p=(t,s,r)\in\ctIJK$ implies $(t,s)\in\ctIJ$
and $(s,r)\in\ctJK$.

If $p=(t,s,r)\in\ctIJK$ has children, the definition implies
$(t,s)\in\ctIJ\setminus\lfIJ$ and $(s,r)\in\ctJK\setminus\lfJK$,
i.e., $\mathcal{P}_{tr}\neq\emptyset$.

On the other hand, $p$ is a leaf if $(t,s)\in\lfIJ$ or $(s,r)\in\lfJK$,
and we can define the \emph{admissible} leaves as
\begin{equation*}
  \lfaIJK := \{ p=(t,s,r)\in\lfIJK\ :\ (t,s)\in\lfaIJ
               \text{ or } (s,r)\in\lfaJK \}.
\end{equation*}
For these admissible leaves $p=(t,s,r)\in\lfaIJK$, the products
$X|_{\hat t\times\hat s} Y|_{\hat s\times\hat r}$ are left or
right semi-uniform (\ref{eq:semiuniform}), and the same holds for
sums of these products, since semi-uniform matrices are a subspace.
$\mathcal{P}_{tr}=\emptyset$ holds if and only if all products
$(t,s,r)\in\ctIJK$ are admissible leaves.

Restricting the product tree to the first and last clusters,
i.e., taking the block $(t,r)$ for every $(t,s,r)\in\ctIJK$, yields
a block tree that is ideally suited for representing the product
$Z = X Y$.

%
%
\begin{definition}[Induced block tree]
Let $\ctIJK$ be a product tree for $\ctIJ$ and $\ctJK$.
Let $\mathcal{T}$ be a block tree for $\ctI$ and $\ctK$.

$\mathcal{T}$ is called the \emph{induced block tree} of $\ctIJK$
if for every $b\in\mathcal{T}$ there is a $p=(t,s,r)\in\ctIJK$
with $b=(t,r)$.

Its set of leaves is denoted by $\mathcal{L} := \{ b\in\mathcal{T}
\ :\ \chil(b)=\emptyset \}$, its set of admissible leaves is
denoted by
\begin{equation*}
  \mathcal{L}^+ := \{ (t,r)\in\mathcal{L}\ :\ 
    (t,s,r)\in\lfaIJK \text{ for all } s\in\ctJ
    \text{ with } (t,s,r)\in\ctIJK \},
\end{equation*}
and its set of inadmissible leaves by $\mathcal{L}^- := \mathcal{L}
\setminus\mathcal{L}^+$.
\end{definition}

We have $(t,r)\in\mathcal{T}$ if there is an $s\in\ctJ$ with
$(t,s,r)\in\ctIJK$.
$(t,r)$ has children if $(t,s,r)$ has children, and this is the
case if $\mathcal{P}_{tr}\neq\emptyset$.
We conclude that the induced block tree reflects how an accumulator
for the root block has to be split to empty $\mathcal{P}_{tr}$.

If $(t,r)\in\mathcal{L}^+$ is an admissible leaf of the induced
block tree $\mathcal{T}^+$, we have $(t,s,r)\in\lfaIJK$ for all
$s\in\ctJ$ with $(t,s,r)\in\ctIJK$, and therefore either
$(t,s)\in\lfaIJ$ or $(s,r)\in\lfaJK$.
This implies that $X|_{\hat t\times\hat s} Y|_{\hat s\times\hat r}$
is a semi-uniform matrix, and since semi-uniform matrices are
a subspace of $\bbbr^{\hat t\times\hat r}$, the entire submatrix
$(X Y)|_{\hat t\times\hat r}$ is also semi-uniform.

If, on the other hand, $(t,r)\in\mathcal{L}^-$ is an inadmissible
leaf of the induced block tree, we can find $s\in\ctJ$ with
$p:=(t,s,r)\in\ctIJK$ and $\chil(p)=\emptyset$, which implies
$\chil(t,s)=\emptyset$ or $\chil(s,r)=\emptyset$.
In the first case, both $t$ and $s$ are leaves of $\ctI$ and $\ctJ$.
In the second case, both $s$ and $r$ are leaves of $\ctJ$ and $\ctK$.
Usually the leaves of the cluster trees contain only a small
number of indices, therefore the submatrix $(X Y)|_{\hat t\times\hat r}$
has a small number of rows or columns and is therefore a low-rank
matrix.

We conclude that we can represent the \emph{exact} product
$X Y$ of two $\mathcal{H}^2$-matrices $X$ and $Y$ as a blockwise
low-rank matrix, where the block structure is defined by the
induced block tree $\mathcal{T}$.

\section{Basis trees}

We take a closer look at the semi-uniform matrices (\ref{eq:semiuniform})
appearing during the construction of the matrix product.
So far, for admissible blocks $(s,r)\in\lfaJK$, we have considered right
semi-uniform matrices of the form
\begin{equation*}
  X|_{\hat t\times\hat s} Y|_{\hat s\times\hat r}
  = X|_{\hat t\times\hat s} V_{Y,s} S_{Y,sr} W_{Y,r}^*
  = A_{tr} W_{Y,r}^*,
\end{equation*}
which is valid for \emph{any} matrix $X|_{\hat t\times\hat s}$.
In our case, $X|_{\hat t\times\hat s}$ is an $\mathcal{H}^2$-matrix,
i.e., its admissible submatrices are expressed by the cluster basis
$V_X$, and we can take advantage of this property to make handling
the matrix $A_{tr}$ more efficient.

%
%
\subsection{Definition}

If $(t,s)\in\lfaIJ$ is an admissible leaf, we have
$X|_{\hat t\times\hat s} = V_{X,t} S_{X,ts} W_{X,s}^*$ and therefore
\begin{equation*}
  X|_{\hat t\times\hat s} Y|_{\hat s\times\hat r}
  = V_{X,t} S_{X,ts} W_{X,s}^* V_{Y,s} S_{Y,sr} W_{Y,r}^*.
\end{equation*}
All of the \emph{basis product matrices} $P_{XY,s} := W_{X,s}^*
V_{Y,s}\in\bbbr^{k\times k}$ can be computed by an efficient recursive
procedure in $\mathcal{O}(n k^2)$ operations \cite[Algorithm~13]{BO10},
and once we have these matrices at our disposal, we can compute
$S_{XY,tr} := S_{X,ts} P_{XY,s} S_{Y,sr}$ in $\mathcal{O}(k^3)$ operations
to obtain
\begin{equation*}
  X|_{\hat t\times\hat s} Y|_{\hat s\times\hat r}
  = V_{X,t} S_{XY,tr} W_{Y,r}^*.
\end{equation*}
Let us consider a slightly more general case:
we assume $(s,r)\in\lfaIJ$, $(t,s)\in\ctIJ\setminus\lfIJ$ with
$\chil(t)=\{t_1,t_2\}$, $\chil(s)=\{s_1,s_2\}$ and
$(t_i,s_j)\in\lfaIJ$ for all $i,j\in\{1,2\}$.
Using the transfer matrices (\ref{eq:transfer}) for $V_{Y,s}$, we find
\begin{align*}
  X|_{\hat t\times\hat s} Y|_{\hat s\times\hat r}
  &= \begin{pmatrix}
      X|_{\hat t_1\times\hat s_1} & X|_{\hat t_1\times\hat s_2}\\
      X|_{\hat t_2\times\hat s_1} & X|_{\hat t_2\times\hat s_2}
    \end{pmatrix}
    V_{Y,s} S_{Y,sr} W_{Y,r}^*\\
  &= \begin{pmatrix}
      V_{X,t_1} S_{X,t_1s_1} W_{X,s_1}^* &
      V_{X,t_1} S_{X,t_1s_2} W_{X,s_2}^*\\
      V_{X,t_2} S_{X,t_2s_1} W_{X,s_1}^* &
      V_{X,t_2} S_{X,t_2s_2} W_{X,s_2}^*
    \end{pmatrix}
    \begin{pmatrix}
      V_{Y,s_1} E_{Y,s_1}\\
      V_{Y,s_2} E_{Y,s_2}
    \end{pmatrix} S_{Y,sr} W_{Y,r}^*\\
  &= \begin{pmatrix}
      V_{X,t_1} (S_{X,t_1s_1} P_{XY,s_1} E_{Y,s_1}
                + S_{X,t_1s_2} P_{XY,s_2} E_{Y,s_2}) S_{Y,sr}\\
      V_{X,t_2} (S_{X,t_2s_1} P_{XY,s_1} E_{Y,s_1}
                + S_{X,t_2s_2} P_{XY,s_2} E_{Y,s_2}) S_{Y,sr}
    \end{pmatrix} W_{Y,r}^*\\
  &= \begin{pmatrix}
       V_{X,t_1} C_{X,t_1}\\
       V_{X,t_2} C_{X,t_2}
     \end{pmatrix} W_{Y,r}^*
\end{align*}
with the $k\times k$ matrices
\begin{align*}
  C_{X,t_1} &:= (S_{X,t_1s_1} P_{XY,s_1} E_{Y,s_1}
                + S_{X,t_1s_1} P_{XY,s_2} E_{Y,s_2}) S_{Y,sr},\\
  C_{X,t_2} &:= (S_{X,t_2s_1} P_{XY,s_1} E_{Y,s_1}
                + S_{X,t_2s_2} P_{XY,s_2} E_{Y,s_2}) S_{Y,sr},
\end{align*}
i.e., we can represent the matrix $A_{tr}$ of the right semi-uniform
matrix $A_{tr} W_{Y,r}^*$ by just two $k\times k$ matrices instead
of one $|\hat t|\times k$ matrix.
Frequently we will have $k\ll |\hat t|$, so this new representation
may be far more efficient than the previous one.

In general, the matrix $X|_{\hat t\times\hat s}$ can have a considerably more
complex structure than in our example, since it may include an
arbitrary number of nested submatrices.
A tree-like structure is suitable to handle these cases.

We do not use the standard graph-theoretical definition of a tree
in this case, but a recursive one:
a tree can be either a leaf or a tuple of sub-trees.
For ease of notation, we include matrices like $C_{X,t}$ directly
in our definition.
We also want the basis tree to correspond to a sub-tree of the
cluster tree $\ctI$, so the number of children of a basis tree
is defined via the number of children of the associated cluster
in the cluster tree.

%
%
\begin{definition}[Basis tree]
For all $t\in\ctI$ with $m=|\chil(t)|$, $\chil(t)=\{t_1,\ldots,t_m\}$,
we define recursively
\begin{equation*}
  \mathcal{B}_t
  = \begin{cases}
    (\bbbr^{k\times k} \times \bbbr^{k\times k})
    \cup
    (\bbbr^{k\times k} \times \bbbr^{k\times k}
    \times \mathcal{B}_{t_1} \times \ldots \times \mathcal{B}_{t_m})
    &\text{ if } m > 0,\\
    \bbbr^{k\times k} \times \bbbr^{\hat t\times k} \times \bbbr^{k\times k}
    &\text{ if } m = 0.
  \end{cases}
\end{equation*}
\end{definition}

A basis tree $\alpha\in\mathcal{B}_t$ can take one of three forms:
\begin{itemize}
  \item If $t\in\ctI$ is a leaf in $\ctI$, we have
    $\alpha = (C,N,M)$ with the coefficient matrix
    $C\in\bbbr^{k\times k}$ and the nearfield matrix
    $N\in\bbbr^{\hat t\times k}$.
  \item If $t\in\ctI$ is not a leaf in $\ctI$, we can have
    $\alpha = (C,M)$ with the coefficient matrix $C\in\bbbr^{k\times k}$, or
  \item we can have $\alpha = (C,M,\alpha_1,\ldots,\alpha_m)$
    with $\alpha_i\in\mathcal{B}_{t_i}$, $i\in\{1,\ldots,m\}$ and the
    coefficient matrix $C\in\bbbr^{k\times k}$.
    In this case we call $\alpha_1,\ldots,\alpha_m$ the \emph{children}
    of $\alpha$.
\end{itemize}
The matrix $M\in\bbbr^{k\times k}$ appearing in all three cases denotes
a transformation that is applied from the right to the matrix
represented by $\alpha$.
Its purpose will become clear when we use basis trees to represent
submatrices of the product.

We use basis trees to represent matrices resulting from the
multiplication of a cluster basis $V_{Y,s}$ with an
$\mathcal{H}^2$-matrix $X$.
In particular, we assign each $\alpha\in\mathcal{B}_t$ a matrix
$\bbbr^{\hat t\times k}$ that we call its \emph{yield}.

%
%
\begin{definition}[Yield of a basis tree]
For all $t\in\ctI$ with $m=|\chil(t)|$, $\chil(t)=\{t_1,\ldots,t_m\}$
and all $\alpha\in\mathcal{B}_t$ we define recursively
\begin{align*}
  \yield(\alpha) &= \begin{cases}
    (V_{X,t} C + N) M &\text{ if } \alpha=(C,N,M),\\
    V_{X,t} C M &\text{ if } \alpha=(C,M),\\
    \begin{pmatrix}
      V_{X,t_1} E_{X,t_1} C + \yield(\alpha_1)\\
      \vdots\\
      V_{X,t_m} E_{X,t_m} C + \yield(\alpha_m)
    \end{pmatrix} M &\text{ if } \alpha=(C,M,\alpha_1,\ldots,\alpha_m).
  \end{cases}
\end{align*}
Due to Definition~\ref{de:cluster_tree}, we have
$\yield(\alpha)\in\bbbr^{\hat t\times k}$ for all $\alpha\in\mathcal{B}_t$.
\end{definition}

To obtain an intuitive understanding of the yield, assume $M=I$ and
extend the columns of $V_{X,t}$ by zero.
With this simplification, the yield of a basis tree $\alpha$ is just the
sum of the products $V_{X,t} C$ for all of its nodes $t$, plus the
zero-extended nearfield matrices $N$ if $t$ is a leaf.

%
%
\subsection{Product representation}

Given a block $b=(t,s)\in\ctIJ$ and a coupling matrix $S\in\bbbr^{k\times k}$,
our goal is to construct a basis tree $\alpha\in\mathcal{B}_t$ with
\begin{equation*}
  \yield(\alpha) = X|_{\hat t\times\hat s} V_{Y,s} S,
\end{equation*}
since this would give us a more efficient representation of
the matrix $A_{tr}$ appearing in (\ref{eq:right_semiuniform}).
Of course, a similar approach can be used for the matrix
$B_{tr}$ in (\ref{eq:left_semiuniform}).

If $(t,s)\in\lfiIJ$, our construction of the block tree $\ctIJ$
implies $t\in\lfI$ and we can simply use
\begin{equation}\label{eq:inadmbasistree}
  \alpha = \inadmbasistree(X|_{\hat t\times\hat s} V_{Y,s} S)
  := (0,X|_{\hat t\times\hat s} V_{Y,s} S,I) \in\mathcal{B}_t
\end{equation}
to obtain $\yield(\alpha) = X|_{\hat t\times\hat s} V_{Y,s} S$.

If $(t,s)\in\lfaIJ$, we have $X|_{\hat t\times\hat s} = V_{X,t}
S_{X,ts} W_{X,s}^*$ and can use
\begin{equation}\label{eq:admbasistree}
  \alpha = \admbasistree(S_{X,ts} P_{XY,s} S)
  := (S_{X,ts} P_{XY,s} S,0,I)\in\mathcal{B}_t
\end{equation}
to obtain again $\yield(\alpha) = V_{X,t} S_{X,ts} P_{XY,s}^* S
= X|_{\hat t\times\hat s} V_{Y,s} S$.
Here the products $P_{XY,s} = W_{X,s}^* V_{Y,s}$ are again assumed to have
been precomputed in advance.

The challenging case is $(t,s)\in\ctIJ\setminus\lfIJ$, i.e., if
$(t,s)$ is subdivided.
For ease of presentation, we restrict here to a $2\times 2$ block structure,
i.e., we assume $\chil(t)=\{t_1,t_2\}$ and $\chil(s)=\{s_1,s_2\}$ and
have the equation
\begin{equation*}
  X|_{\hat t\times\hat s}
  = \begin{pmatrix}
      X|_{\hat t_1\times\hat s_1} &
      X|_{\hat t_1\times\hat s_2}\\
      X|_{\hat t_2\times\hat s_1} &
      X|_{\hat t_2\times\hat s_2}
    \end{pmatrix}.
\end{equation*}
Using the transfer matrices (\ref{eq:transfer}), we find
\begin{equation*}
  V_{Y,s} = \begin{pmatrix}
    V_{Y,s_1} E_{Y,s_1}\\
    V_{Y,s_2} E_{Y,s_2}
  \end{pmatrix}
\end{equation*}
and therefore
\begin{align*}
  X|_{\hat t\times\hat s} V_{Y,s} S_{Y,sr}
  &= \begin{pmatrix}
      X|_{\hat t_1\times\hat s_1} &
      X|_{\hat t_1\times\hat s_2}\\
      X|_{\hat t_2\times\hat s_1} &
      X|_{\hat t_2\times\hat s_2}
    \end{pmatrix}
    \begin{pmatrix}
      V_{Y,s_1} E_{Y,s_1} S\\    
      V_{Y,s_2} E_{Y,s_2} S
    \end{pmatrix}\\
  &= \begin{pmatrix}
      X|_{\hat t_1\times\hat s_1} V_{Y,s_1} E_{Y,s_1} S
      + X|_{\hat t_1\times\hat s_2} V_{Y,s_2} E_{Y,s_2} S\\
      X|_{\hat t_2\times\hat s_1} V_{Y,s_1} E_{Y,s_1} S
      + X|_{\hat t_2\times\hat s_2} V_{Y,s_2} E_{Y,s_2} S
     \end{pmatrix}.
\end{align*}
We can see that each summand has the same shape as the
original matrix, so we assume that we can apply recursion to
find basis trees $\alpha_{11},\alpha_{12}\in\mathcal{B}_{t_1}$ and
$\alpha_{21},\alpha_{22}\in\mathcal{B}_{t_2}$ with
\begin{align*}
  \yield(\alpha_{11})
  &= X|_{\hat t_1\times\hat s_1} V_{Y,s_1} E_{Y,s_1} S, &
  \yield(\alpha_{12})
  &= X|_{\hat t_1\times\hat s_2} V_{Y,s_2} E_{Y,s_2} S,\\
  \yield(\alpha_{21})
  &= X|_{\hat t_2\times\hat s_1} V_{Y,s_1} E_{Y,s_1} S, &
  \yield(\alpha_{22})
  &= X|_{\hat t_2\times\hat s_2} V_{Y,s_2} E_{Y,s_2} S.
\end{align*}
To match our definition, we have to find $\alpha_1\in\mathcal{B}_{t_1}$
and $\alpha_2\in\mathcal{B}_{t_2}$ with
\begin{align*}
  \yield(\alpha_1)
  &= \yield(\alpha_{11}) + \yield(\alpha_{12}), &
  \yield(\alpha_2)
  &= \yield(\alpha_{21}) + \yield(\alpha_{22}),
\end{align*}
because then $\alpha := (0, I, \alpha_1, \alpha_2) \in \mathcal{B}_t$
would satisfy $\yield(\alpha) = X|_{\hat t\times\hat s} V_{Y,s} S$ by
definition.

In order to achieve this goal, we have to make sure that the
structures of $\alpha_{11}$ and $\alpha_{12}$ as well as
$\alpha_{21}$ and $\alpha_{22}$ match, because then adding them
together is straightforward.

To this end, we introduce operations on basis trees that change
the representation without changing the yield.
We use the auxiliary function
\begin{align*}
  \trans(\alpha, X)
  &:= \begin{cases}
    (C, N, MX) &\text{ if } \alpha=(C,N,M),\\
    (C, MX) &\text{ if } \alpha=(C,M),\\
    (C, MX, \alpha_1, \ldots, \alpha_m)
    &\text{ if } \alpha=(C,M,\alpha_1,\ldots,\alpha_m)
  \end{cases}
  &\text{ for } \alpha\in\mathcal{B}_t
\end{align*}
to apply a transformation to a basis tree such that
\begin{align*}
  \yield(\trans(\alpha, X))
  &= \yield(\alpha) X &
  &\text{ for all } \alpha\in\mathcal{B}_t,\ X\in\bbbr^{k\times k}.
\end{align*}
If we want to work with the matrices $C$ or $N$, we have to get
rid of potentially pending transformations.
This can be achieved by the function
\begin{align*}
  \finish(\alpha)
  &:= \begin{cases}
    (CM, NM, I) &\text{ if } \alpha=(C,N,M),\\
    (CM, I) &\text{ if } \alpha=(C,M),\\
    (CM, I, \trans(\alpha_1,M), \ldots, \trans(\alpha_m,M))
    &\text{ if } \alpha=(C,M,\alpha_1,\ldots,\alpha_m)
    \end{cases}
\end{align*}
for $\alpha\in\mathcal{B}_t$ that replaces $C$ and $N$ with
their transformed counterparts and shifts the transformation
towards the children if children exist.

Adding two basis tree representations together is particularly
easy if both have the same structure.
If $t\in\lfI$, we can simply use
\begin{align*}
  \add(\alpha,\beta)
  &:= (C_\alpha M_\alpha + C_\beta M_\beta,
       N_\alpha M_\alpha + N_\beta M_\beta, I)
\end{align*}
for $\alpha=(C_\alpha,N_\alpha,M_\alpha),
\beta=(C_\beta,N_\beta,M_\beta)\in\mathcal{B}_t$.

If $t\in\ctI\setminus\lfI$, we can handle the case that
both $\alpha$ and $\beta$ are of the same type easily,
i.e., if $\alpha=(C_\alpha,M_\alpha)$ and $\beta=(C_\beta,M_\beta)$,
we use
\begin{equation*}
  \add(\alpha,\beta)
  := (C_\alpha M_\alpha + C_\beta M_\beta, I),
\end{equation*}
while for $\alpha=(C_\alpha,M_\alpha,\alpha_1,\ldots,\alpha_m),
\beta=(C_\beta,M_\beta,\beta_1,\ldots,\beta_m)\in\mathcal{B}_t$
we proceed by recursion via
\begin{align*}
  \add(\alpha,\beta)
  &:= (C_\alpha M_\alpha + C_\beta M_\beta, I,\\
  &\qquad \add(\trans(\alpha_1,M_\alpha),\trans(\beta_1,M_\beta)),\ldots,
      \add(\trans(\alpha_m,M_\alpha),\trans(\beta_m,M_\beta))),
\end{align*}
i.e., we add the coefficient matrices and all sub-trees, taking
pending transformations into account.

The case $\alpha=(C_\alpha,M_\alpha)$ and $\beta=(C_\beta,M_\beta,
\beta_1,\ldots,\beta_m)$ requires a little more work: we
``split'' $\alpha$ into sub-trees
\begin{align*}
  \alpha_i
  &:= \begin{cases}
        (E_{X,t_i} C_\alpha M_\alpha, 0, I) &\text{ if } t_i\in\lfI,\\
        (E_{X,t_i} C_\alpha M_\alpha, I) &\text{ otherwise}
      \end{cases} &
  &\text{ for all } i\in\{1,\ldots,m\}
\end{align*}
and define
\begin{equation*}
  \bsplit(\alpha) := (0, I, \alpha_1, \ldots, \alpha_m).
\end{equation*}
Due to the transfer matrix equation (\ref{eq:transfer}), we find
\begin{align*}
  \yield(\bsplit(\alpha))
  &= \yield(0, I, \alpha_1,\ldots,\alpha_m)\\
  &= \begin{pmatrix}
       V_{X,t_1} E_{X,t_1} C_\alpha M_\alpha\\
       \vdots\\
       V_{X,t_m} E_{X,t_m} C_\alpha M_\alpha
    \end{pmatrix}
   = V_{X,t} C_\alpha M_\alpha
   = \yield(\alpha),
\end{align*}
i.e., $\bsplit(\alpha)$ is an alternative representation of $\alpha$
that shares the structure of $\beta$ so that we have
\begin{align*}
  \yield(\alpha) + \yield(\beta)
  &= \yield(\bsplit(\alpha)) + \yield(\beta)\\
  &= \yield(0,I,\alpha_1,\ldots,\alpha_m)
   + \yield(C_\beta,M_\beta,\beta_1,\ldots,\beta_m)
\end{align*}
and can proceed as in the previous case to use
\begin{equation*}
  \add(\alpha,\beta)
  := \add(\bsplit(\alpha),\beta)
\end{equation*}
to obtain $\yield(\add(\alpha,\beta))=\yield(\alpha)+\yield(\beta)$.

%
%
\begin{figure}
  \begin{quotation}
    \begin{tabbing}
      \textbf{procedure} addproduct($X|_{\hat t\times\hat s}$, $S$,
                                    \textbf{var} $\alpha$);\\
      \textbf{begin}\\
      \quad\= \textbf{if} $(t,s)\in\lfaIJ$ \textbf{then begin}\\
      \> \quad\= $\beta \gets \admbasistree(S_{X,ts} P_{XY,s} S)$;
          \qquad $\alpha \gets \add(\alpha,\beta)$\\
      \> \textbf{end else if} $(t,s)\in\lfiIJ$ \textbf{then begin}\\
      \> \> $\beta \gets \inadmbasistree(X|_{\hat t\times\hat s} V_{Y,s} S)$;
          \qquad $\alpha \gets \add(\alpha,\beta)$\\
      \> \textbf{end else begin}\\
      \> \> \textbf{if} $\alpha$ not split \textbf{then}\\
      \> \> \quad\= $\alpha \gets \bsplit(\alpha)$;\\
      \> \> \textbf{for} $t'\in\chil(t)$,\ $s'\in\chil(s)$ \textbf{do}\\
      \> \> \> addproduct($X|_{\hat t'\times\hat s'}$, $E_{Y,s'} S$,
                          $\alpha_{t'}$)\\
      \> \textbf{end}\\
      \textbf{end}
    \end{tabbing}
  \end{quotation}
  \caption{Adding a product $X|_{\hat t\times\hat s} V_{Y,s} S$ to a
           basis tree}
  \label{fi:addproduct}
\end{figure}

Now we have everything at our disposal to construct a basis tree
for the product $X|_{\hat t\times\hat s} V_{Y,s} S$.
The resulting algorithm is summarized in Figure~\ref{fi:addproduct}.
For an admissible block $(t,s)\in\lfaIJ$, we use the function
``uniform'' from (\ref{eq:admbasistree}) to construct $\beta\in\mathcal{B}_t$
with $\yield(\beta)=X|_{\hat t\times\hat s} V_{Y,s} S$ and add it to $\alpha$.
For an inadmissible leaf $(t,s)\in\lfiIJ$, we use the function
``nearfield'' from (\ref{eq:inadmbasistree}) instead.
If $(t,s)\in\ctIJ\setminus\lfIJ$ is subdivided, we use the function
``split'' to make sure that $\alpha$ is also subdivided and then
add the submatrices of the block to the children of $\alpha$ recursively.

%
%
\begin{figure}
  \pgfdeclareimage[width=11cm]{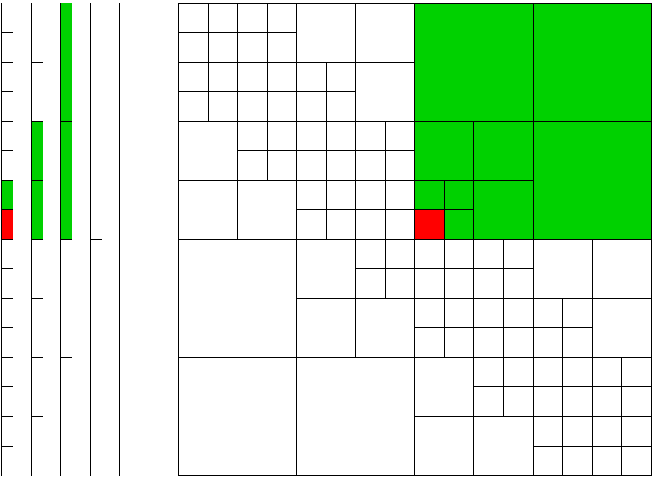}{minimaltree}
  \begin{center}
  \begin{pgfpicture}{0cm}{0cm}{11cm}{8.0cm}
    \pgfuseimage{minimaltree}
  \end{pgfpicture}
  \end{center}
  \caption{Minimal (row) basis tree for a submatrix: The coloured matrix
  blocks on the right correspond to the coloured row clusters on the left.}
  \label{fi:minimaltree}
\end{figure}

Figure~\ref{fi:minimaltree} shows the result of this procedure when
applied to the (green and red) right upper quadrant of an
$\mathcal{H}^2$-matrix.
We can see that the algorithm ensures that a cluster $t$ appears in the
basis tree if and only if there is a block $(t,s)$ in the submatrix.

%
%
\begin{remark}[Complexity]
It is important to keep in mind that the functions ``uniform'', ``nearfield'',
and ``add'' require only $\mathcal{O}(k^3)$ operations, and the
same holds for ``split'', assuming that a cluster only has a bounded
number of children.

Since the function ``addproduct'' in Figure~\ref{fi:addproduct} calls itself
recursively only if the current block $(t,s)$ is subdivided, we may conclude
that this function requires only $\mathcal{O}(k^3)$ operations for every
descendant of the block $(t,s)$.

Given that the function constructs an \emph{exact} representation
of the product, this may be considered the optimal complexity.
\end{remark}

%
%
\subsection{Basis tree accumulator}

With efficient basis trees at our disposal, we can turn back to
the improvement of the accumulators introduced in
Definition~\ref{de:accumulator}.

We have seen that we can represent the matrices $A_{tr}$ and
$B_{tr}$ of an accumulator by basis trees, i.e.,
\begin{align*}
  A_{tr} &= \yield(\alpha_{tr}), &
  B_{tr} &= \yield(\beta_{tr}),
\end{align*}
where $\alpha_{tr}\in\mathcal{B}_t$ and $\beta_{tr}\in\mathcal{B}_r$.
We already know that performing the updates
\begin{align*}
  A_{tr} &\gets A_{tr} + X|_{\hat t\times\hat s} V_{Y,s} S_{Y,sr}, &
  B_{tr} &\gets B_{tr} + Y|_{\hat s\times\hat r}^* W_{Y,s} S_{X,ts}^*
\end{align*}
requires only $\mathcal{O}(|\desc(t,s)|\,k^3)$ and
$\mathcal{O}(|\desc(s,r)|\,k^3)$ operations, respectively,
using this special representation.

We still need an efficient way of splitting an accumulator.
Handling the nearfield $N_{tr}$ and the subdivided products
$\mathcal{P}_{tr}$ can remain as before, but the special
representation of $A_{tr}$ and $B_{tr}$ requires a new approach.

We consider a splitting step applied to an accumulator for
the submatrix $(t,r)$ with children $(t',r')$, $t'\in\chil(t)$,
$r'\in\chil(r)$.
For the right-semiuniform part we require
\begin{equation*}
  A_{t'r'} = A_{tr}|_{\hat t'\times k} F_{Y,r'}^*,
\end{equation*}
while the left-semiuniform part leads to
\begin{equation*}
  B_{t'r'} = B_{tr}|_{\hat r'\times k} E_{X,r'}^*.
\end{equation*}
We have already designed the basis tree representation in order
to make transformations efficient, therefore we can use
\begin{align*}
  A_{tr} F_{Y,r'}^*
  &= \yield(\alpha_{tr}) F_{Y,r'}^*
   = \yield(\trans(\alpha_{tr}, F_{Y,r'}^*)),\\
  B_{tr} E_{X,t'}^*
  &= \yield(\beta_{tr}) E_{X,t'}^*
   = \yield(\trans(\beta_{tr}, E_{X,t'}^*))
\end{align*}
to obtain intermediate representations in $\mathcal{O}(k^3)$
operations.

This leaves us with the task of restricting $A_{tr}$ to the
rows in $\hat t'$ for a child $t'\in\chil(t)$.
If $t$ is a leaf, no such child exists, so we only have to
consider $\alpha_{tr} = (C, M)$ and
$\alpha_{tr} = (C, M, \alpha_1, \ldots, \alpha_m)$.

The first case is straightforward: we define
\begin{equation*}
  \alpha_{t'r} := \begin{cases}
     (E_{X,t'} C M, 0, I) &\text{ if } t'\in\lfI,\\
     (E_{X,t'} C M, I) &\text{ otherwise}
  \end{cases}
\end{equation*}
and use (\ref{eq:transfer}) to obtain
\begin{equation*}
  \yield(\alpha_{t'r}) = V_{X,t'} E_{X,t'} C M
  = V_{X,t}|_{\hat t'\times k} C M
  = \yield(\alpha_{tr})|_{\hat t'\times k}
  = A_{tr}|_{\hat t'\times k}.
\end{equation*}
The second case $\alpha_{tr} = (C, M, \alpha_1, \ldots, \alpha_m)$
poses a challenge if $C\neq 0$, since $\yield(\alpha_{tr})$ is
influenced by $C$, but the yields of children are not.
Once again we can use (\ref{eq:transfer}) to solve this problem:
$C$ contributes to $\yield(\alpha_{tr})$ via
\begin{equation*}
  V_{X,t} C M,
\end{equation*}
and this product's restriction to the rows in $t'$ can be expressed
by
\begin{equation*}
  (V_{X,t} C M)|_{\hat t'\times k}
  = V_{X,t}|_{\hat t'\times k} C M
  = V_{X,t'} E_{X,t'} C M.
\end{equation*}
This allows us to add the contribution of the root to the coefficient
matrices of the children.
We know that there is an $i\in\{1,\ldots,m\}$ with
$\alpha_i\in\mathcal{B}_{t'}$ by definition, and we modify $\alpha_i$
to take care of the ancestor's contribution:
\begin{align*}
  \hat\alpha_i &:= \finish(\trans(\alpha_i, M)),\\
  \alpha_{t'r} &:= \begin{cases}
    (C_i + E_{X,t'} C M, I) &\text{ if } \hat\alpha_i = (C_i, I),\\
    (C_i + E_{X,t'} C M, N_i, I) &\text{ if } \hat\alpha_i = (C_i, N_i, I),\\
    (C_i + E_{X,t'} C M, I, \alpha_{i1}, \ldots, \alpha_{in})
    &\text{ if } \hat\alpha_i = (C_i, I, \alpha_{i1}, \ldots, \alpha_{in}).
  \end{cases}
\end{align*}
In the first step, we apply pending transformations $M$ of $\alpha$
and $M_i$ of $\alpha_i$ to obtain $\hat\alpha_i$, in the second step,
we then add the contribution of the parent's coefficient matrix
to the child.
As we have seen before, this guarantees
\begin{equation*}
  \yield(\alpha_{t'r})
  = \yield(\alpha_i) + V_{X,t'} E_{x,t'} C M
  = \yield(\alpha_{tr})|_{\hat t'\times k}
  = A_{tr}|_{\hat t'\times k}.
\end{equation*}
Once again we observe that $\mathcal{O}(k^3)$ operations are sufficient
to construct $A_{t'r'}$ from $A_{tr}$.
By the same arguments, we can also construct $B_{t'r}$ from $B_{tr}$ and
conclude that splitting an accumulator in basis tree representation
requires only $\mathcal{O}(k^3)$ instead of
$\mathcal{O}((|\hat t|+|\hat r|)k^2)$ operations.

%
%
\begin{remark}[Implementation]
The key advantage of the basis tree representation is that most algorithms
require only $\mathcal{O}(k^3)$ operations, i.e., to apply a transformation,
to accumulate transformations, to add an admissible or inadmissible block,
or to restrict to a submatrix.

In a practical implementation, this behaviour can indeed be guaranteed
by using a copy-on-write approach:
instead of changing the representation of a basis tree in place, we always
create and work with copies of the affected nodes, thus avoiding
inconsistencies.
\end{remark}

\section{Adaptive bases}

We have seen that we can obtain semi-uniform representations
of all submatrices of the product in $\mathcal{O}(n k^2 \log n)$
operations, and we are faced with two challenges:
the preliminary representation will generally have a finer
block structure than the one we would like to used, and it
will also consist of semi-uniform submatrices, while we require
an $\mathcal{H}^2$-matrix in order to be able to perform further
operations.

%
%
\subsection{Coarsening}
For the first challenge, we can use a very simple approach:
if we have a block matrix consisting of low-rank submatrices,
a well-known approach called \emph{agglomeration} in
\cite[Section~2.6.2]{HA15} can be used to merge multiple low-rank
sub-matrices into one large low-rank matrix.

Let $t\in\ctI$, $s\in\ctJ$, and $\chil(t)=\{t_1,t_2\}$,
and assume that we have low-rank matrices $A B^*$ and $C D^*$
with $A\in\bbbr^{\hat t_1\times k}$, $B,D\in\bbbr^{\hat s\times k}$,
$C\in\bbbr^{\hat t_2\times k}$ at our disposal.
We are looking for a low-rank approximation of the block matrix
\begin{equation*}
  G := \begin{pmatrix}
    A B^*\\
    C D^*\\
  \end{pmatrix}
  = \begin{pmatrix}
      A & \\
      & C
    \end{pmatrix}
    \begin{pmatrix}
      B & D
    \end{pmatrix}^*
\end{equation*}
and start by computing a thin Householder factorization
\begin{equation*}
  \begin{pmatrix}
    B & D
  \end{pmatrix}
  = Q R = Q \begin{pmatrix} R_1 & R_2 \end{pmatrix}
\end{equation*}
with matrices $R_1,R_2\in\bbbr^{(2k)\times k}$ and an isometric matrix
$Q\in\bbbr^{s\times (2k)}$.
This implies
\begin{equation*}
  G = \begin{pmatrix}
    A B^*\\
    C D^*
  \end{pmatrix}
  = \begin{pmatrix}
    A & \\
    & C
  \end{pmatrix}
  \begin{pmatrix} R_1^*\\ R_2^* \end{pmatrix} Q^*
  = \begin{pmatrix}
    A R_1^*\\
    C R_2^*
  \end{pmatrix} Q^*,
\end{equation*}
and the condensed matrix
\begin{equation*}
  \widehat{G} := \begin{pmatrix}
    A R_1^*\\ C R_2^*
  \end{pmatrix}\in\bbbr^{\hat t\times(2k)}
\end{equation*}
has only $2k$ columns, allowing us to compute its singular value
decomposition
\begin{equation*}
  \widehat{G} = U \Sigma \widehat{V}^*
\end{equation*}
in $\mathcal{O}(|\hat t| k^2)$ operations to obtain
\begin{equation*}
  G = \widehat{G} Q^* = U \Sigma \widehat{V}^* Q^*,
\end{equation*}
a (thin) singular value decomposition of the original matrix $G$.
Computing $V = Q \widehat{V}$ requires only
$\mathcal{O}(|\hat s| k^2)$ operations, so we can obtain the thin
singular value decomposition of $G$ efficiently.
Now truncating to a low-rank approximation is straightforward,
e.g., we can simply drop the smallest singular values and keep
just the first left and right singular vectors in $U$ and $V$.

Whenever the block structure of the intermediate representation
of the product is too fine, we can apply this agglomeration procedure
repeatedly to merge all offending sub-matrices into one large
low-rank matrix.

%
%
\begin{remark}[Error control]
Since we are using singular value decompositions, we can perfectly
control the error introduced in each step \cite[Theorem~2.11]{HA15}.

We choose a relative tolerance $\epsilon\in\bbbr_{>0}$ for the
truncation operations.

Let us consider the agglomeration of a block matrix
\begin{equation*}
  G = \begin{pmatrix}
    G_{11} & \ldots & G_{1m}\\
    \vdots & \ddots & \vdots\\
    G_{n1} & \ldots & G_{nm}
  \end{pmatrix}.
\end{equation*}
If we first sequentially agglomerate every row, i.e.,
\begin{align*}
  \widetilde{G}_{i1} &= G_{i1}, &
  \widetilde{G}_{ij} &\approx \begin{pmatrix}
      \widetilde{G}_{i,j-1} & G_{ij}
  \end{pmatrix} &
  &\text{ for all } i\in[1:n],\ j\in[2:m],
\end{align*}
we can use the singular values to ensure
\begin{align*}
  \|\begin{pmatrix}
      \widetilde{G}_{i,j-1} & G_{ij}
    \end{pmatrix} - \widetilde{G}_{ij}\|_2
  &\leq \epsilon
    \|\begin{pmatrix}
       \widetilde{G}_{i,j-1} & G_{ij}
      \end{pmatrix}\|_2 &
  &\text{ for all } i\in[1:n],\ j\in[2:m],
\end{align*}
and since orthogonal projections cannot increase the spectral norm,
we also have
\begin{align*}
  \|\begin{pmatrix}
      \widetilde{G}_{i,j-1} & G_{ij}
    \end{pmatrix} - \widetilde{G}_{ij}\|_2
  &\leq \epsilon \|\begin{pmatrix}
      G_{i1} & \ldots & G_{ij}
    \end{pmatrix}\|_2 &
  &\text{ for all } i\in[1:n],\ j\in[2:m].
\end{align*}
For the entire row, we can use a telescoping sum to find
\begin{align*}
  \|\begin{pmatrix}
     G_{i1} & \ldots & G_{im}
    \end{pmatrix} - \widetilde{G}_{im} \|_2
  &\leq \epsilon (m-1)
   \|\begin{pmatrix}
      G_{i1} & \ldots & G_{im}
     \end{pmatrix}\|_2 &
  &\text{ for all } i\in[1:n].
\end{align*}
The triangle inequality allows us to apply this estimate to every
block row of $G$ to get the estimate
\begin{equation*}
  \left\|G - \begin{pmatrix} \widetilde{G}_{1m}\\
                        \vdots\\
                        \widetilde{G}_{nm} \end{pmatrix}\right\|_2
  \leq \sum_{i=1}^n \|\begin{pmatrix} G_{i1} & \ldots & G_{im}
                     \end{pmatrix} - \widetilde{G}_{im} \|_2
  \leq \epsilon(m-1) n \|G\|_2.
\end{equation*}
To obtain a low-rank approximation of the entire matrix $G$, we
can now sequentially agglomerate the agglomerated rows
$\widetilde{G}_{1m},\ldots,\widetilde{G}_{nm}$ by the same technique
to find a low-rank matrix $\widetilde{G}$ with
\begin{equation*}
  \left\|\begin{pmatrix} \widetilde{G}_{1m}\\
                         \vdots\\
                         \widetilde{G}_{nm} \end{pmatrix}
         - \widetilde{G}\right\|_2
  \leq \epsilon (n-1)
    \left\|\begin{pmatrix} \widetilde{G}_{1m}\\
                           \vdots\\
                           \widetilde{G}_{nm} \end{pmatrix}
    \right\|_2.
\end{equation*}
The total error can be obtained by the triangle inequality:
\begin{align*}
  \|G - \widetilde{G}\|_2
  &\leq \left\|G - \begin{pmatrix}
                     \widetilde{G}_{1m}\\
                     \vdots\\
                     \widetilde{G}_{nm}
                   \end{pmatrix}\right\|_2
   + \left\| \begin{pmatrix}
               \widetilde{G}_{1m}\\
               \vdots\\
               \widetilde{G}_{nm}
             \end{pmatrix} - \widetilde{G} \right\|_2\\
  &\leq \left\|G - \begin{pmatrix}
                     \widetilde{G}_{1m}\\
                     \vdots\\
                     \widetilde{G}_{nm}
                   \end{pmatrix}\right\|_2
   + \epsilon (n-1) \left\| \begin{pmatrix}
                               \widetilde{G}_{1m}\\
                               \vdots\\
                               \widetilde{G}_{nm}
                            \end{pmatrix} \right\|_2\\
  &\leq \left\|G - \begin{pmatrix}
                     \widetilde{G}_{1m}\\
                     \vdots\\
                     \widetilde{G}_{nm}
                   \end{pmatrix}\right\|_2
   + \epsilon (n-1) \|G\|_2
   + \epsilon (n-1) \left\| G - \begin{pmatrix}
                               \widetilde{G}_{1m}\\
                               \vdots\\
                               \widetilde{G}_{nm}
                            \end{pmatrix} \right\|_2\\
  &\leq \epsilon (m-1) n \|G\|_2
      + \epsilon (n-1) \|G\|_2
      + \epsilon^2 (n-1) (m-1) n \|G\|_2\\
  &= \epsilon \|G\|_2 \bigl((m-1) n + (n-1) + \epsilon (n-1) n (m-1)\bigr)\\
  &\leq \epsilon \|G\|_2 (m n + \epsilon m n^2).
\end{align*}
If the standard admissibility condition is used, the numbers $n$ and $m$
of the block rows and columns are bounded by a small constant.
Given any relative accuracy $\hat\epsilon\in(0,1]$ for the final result,
we can choose the relative truncation tolerance
$\epsilon := \frac{\hat\epsilon}{2 m n}$ in order to obtain
\begin{equation*}
  \epsilon ( m n + \epsilon m n^2 )
  = \frac{\hat\epsilon}{2 m n}
    \left( m n + \frac{\hat\epsilon}{2 m n} m n^2 \right)
  = \frac{\hat\epsilon}{2}
    + \frac{\hat\epsilon^2}{4 m}
  = \hat\epsilon\left( \frac{1}{2} + \frac{\hat\epsilon}{4 m} \right)
  < \hat\epsilon.
\end{equation*}
This allows us to guarantee $\|G-\widetilde{G}\|_2 \leq \hat\epsilon \|G\|_2$,
i.e., the coarsening procedure can provide any block-relative error
we may require.
\end{remark}

%
%
\subsection{$\mathcal{H}^2$-matrix recompression}

The coarsening procedure allows us to obtain an given block structure
for the approximation of the result, but we sacrifice the representation
of the matrices by semi-uniform matrices and basis trees.
In essence, we obtain a standard $\mathcal{H}$-matrix approximation
of the product $XY$ of two $\mathcal{H}^2$-matrices.

If we want to use this product in further $\mathcal{H}^2$-matrix
algorithms, we have to include another approximation step that
restores the $\mathcal{H}^2$-matrix structure with suitable cluster
bases.
Fortunately, there is an efficient algorithm for this task
\cite[Section~6.5]{BO10} that allows us to guarantee a user-defined
accuracy.

Let $G\in\bbbr^{\Idx\times\Jdx}$ denote the hierarchical matrix
approximation of the product $XY$ for the prescribed block tree
$\ctIK$.
We are looking for a row cluster basis that can approximate
all blocks connected to a given row cluster
\begin{align*}
  \brow(t) &:= \{ r\in\ctK\ :\ (t,r)\in\lfaIK \} &
  &\text{ for all } t\in\ctI\\
\intertext{and for a column cluster basis that can approximate
all blocks connected to a given column cluster}
  \bcol(r) &:= \{ t\in\ctI\ :\ (t,r)\in\lfaIK \} &
  &\text{ for all } r\in\ctK.
\end{align*}
This is, however, not all that our cluster bases have to
accomplish: due to the nested structure (\ref{eq:transfer}),
a cluster basis $V_t$ has to be able to approximate not only
all blocks in $\brow(t)$, but also parts of all blocks
connected to the \emph{predecessors} of $t$, since the bases
for these predecessors are defined using transfer matrices
and are therefore not independent from their descendants.

To keep track of \emph{all} blocks we have to approximate,
we introduce the extended sets
\begin{align*}
  \brow^*(t) &:= \begin{cases}
    \brow(t) \cup \brow^*(t^+)
    &\text{ if there is } t^+\in\ctI \text{ with } t\in\chil(t^+),\\
    \brow(t) &\text{ otherwise}
  \end{cases} &
  &\text{ for all } t\in\ctI,\\
  \bcol^*(r) &:= \begin{cases}
    \bcol(r) \cup \bcol^*(r^+)
    &\text{ if there is } r^+\in\ctI \text{ with } r\in\chil(r^+),\\
    \bcol(r) &\text{ otherwise}
  \end{cases} &
  &\text{ for all } r\in\ctK
\end{align*}
that include all column or row clusters, respectively, connected
to a cluster or one of its ancestors.

We focus on the construction of an adaptive row cluster basis,
since the exact same procedure can be applied to the adjoint
matrix $G^*$ to obtain a column cluster basis.
To avoid redundant information, we assume that we are looking for
an isometric cluster basis, i.e., that
\begin{align*}
  V_t^* V_t &= I &
  &\text{ holds for all } t\in\ctI.
\end{align*}
This ensures that the columns of $V_t$ provide an orthonormal
basis of its range and that $V_t V_t^*$ is an orthogonal projection
into this range, so that $V_t V_t^* G|_{\hat t\times\hat s}$ is
the best approximation of $G|_{\hat t\times\hat s}$ within this
range with respect to the Frobenius und spectral norm.

Since we are looking for a row cluster basis that can handle
all admissible blocks equally well, we require
\begin{align*}
  V_t V_t^* G|_{\hat t\times\hat r}
  &\approx G|_{\hat t\times\hat r} &
  &\text{ for all } t\in\ctI,\ r\in\brow^*(t).
\end{align*}
We assume that $G$ is already an $\mathcal{H}$-matrix, i.e., that
for every $r\in\brow(t)$ we have a low-rank factorization
\begin{align*}
  G|_{\hat t\times\hat r}
  &= A_{tr} B_{tr}^*, &
  A_{tr} &\in \bbbr^{\hat t\times k},\ B_{tr}\in\bbbr^{\hat r\times k}
\end{align*}
at our disposal.
We can take advantage of this fact to reduce the computational
work:
we compute a thin QR factorization $B_{tr} = Q_{tr} R_{tr}$ with
an isometric matrix $Q_{tr}\in\bbbr^{\hat r\times k}$ and a
upper triangular matrix $R_{tr}\in\bbbr^{k\times k}$ and observe
\begin{equation*}
  \|G|_{\hat t\times\hat r} - V_t V_t^* G|_{\hat t\times\hat r}\|_2
  = \|(A_{tr} R_{tr}^* - V_t V_t^* A_{tr} R_{tr}^*) Q_{tr}^*\|_2
  = \|A_{tr} R_{tr}^* - V_t V_t^* A_{tr} R_{tr}^*\|_2,
\end{equation*}
since right-multiplication by an adjoint isometric matrix does
not change the spectral or Frobenius norm, i.e., we can replace the
admissible blocks $G|_{\hat t\times\hat r}$ by ``condensed'' matrices
$G^c_{tr} := A_{tr} R_{tr}^*$ with only $k$ columns without changing
the approximation quality.

Now that we know what we are looking for, we can find an appropriate
algorithm.
If $t$ is a leaf cluster, we have to construct an isometric matrix $V_t$
directly such that
\begin{align}\label{eq:Vt_approx}
  V_t V_t^* G_{tr}^c &\approx G_{tr}^c &
  &\text{ for all } r\in\brow^*(t).
\end{align}
This task can be solved by combining all admissible blocks in a large
matrix
\begin{align*}
  G_t &:= \begin{pmatrix}
    G_{tr_1}^c & \ldots & G_{tr_m}^c
  \end{pmatrix}, &
  \brow^*(t) &= \{ r_1,\ldots,r_m \},
\end{align*}
computing its singular value decomposition and using the first $k$
left singular vectors to form the columns of $V_t$.
The resulting error is given by
\begin{equation*}
  \|G_t - V_t V_t^* G_t\|_2 = \sigma_{k+1},
\end{equation*}
where $\sigma_{k+1}$ is the largest singular value dropped.
By choosing the rank $k$ appropriately, we can ensure that any required
accuracy is guaranteed.

If $t$ is not a leaf cluster, Definition~\ref{de:cluster_basis} implies
that $V_t$ depends on the bases $V_{t'}$ for the children $t'\in\chil(t)$,
so it is advisable to compute $V_{t'}$ first.

To keep the presentation simple, we focus on the case of two children
$\chil(t)=\{t_1,t_2\}$ and remark that the extension to the general
case is straightforward.
Once isometric matrices $V_{t_1}$ and $V_{t_2}$ for the two children
have been computed, (\ref{eq:transfer}) implies
\begin{align*}
  V_t &= \begin{pmatrix}
    V_{t_1} E_{t_1}\\
    V_{t_2} E_{t_2}
  \end{pmatrix}
  = \begin{pmatrix}
    V_{t_1} & \\
    & V_{t_2}
  \end{pmatrix}
  \begin{pmatrix}
    E_{t_1}\\
    E_{t_2}
  \end{pmatrix}
  = U_t \widehat{V}_t
\end{align*}
with the matrices
\begin{align}\label{eq:Ut_Vhatt}
  U_t &:= \begin{pmatrix}
    V_{t_1} & \\
    & V_{t_2}
  \end{pmatrix}, &
  \widehat{V}_t := \begin{pmatrix}
    E_{t_1}\\
    E_{t_2}
  \end{pmatrix}.
\end{align}
The isometric matrix $U_t$ is already defined by the children, and
since we want $V_t$ to be isometric, too, we have to ensure
$\widehat{V}_t^* \widehat{V}_t = I$.

To determine $\widehat{V}_t$, we take a look at the required
approximation properties:
we need (\ref{eq:Vt_approx}) for all $r\in\brow^*(t)$, i.e., the error
\begin{align}
  \|G_{tr}^c - V_t V_t^* G_{tr}^c\|_2^2\label{eq:error_decomp}
  &= \|G_{tr}^c - U_t U_t^* G_{tr}^c
       + U_t U_t^* G_{tr}^c - U_t (\widehat{V}_t \widehat{V}_t^*) U_t^*
         G_{tr}^c\|_2^2\\
  &= \|G_{tr}^c - U_t U_t^* G_{tr}^c\|_2^2
     + \|\widehat{G}_{tr}^c - \widehat{V}_t \widehat{V}_t^*
         \widehat{G}_{tr}^c\|_2^2
  \qquad\text{ with } \widehat{G}_{tr}^c := U_t^* G_{tr}^c\notag
\end{align}
has to be under control for all $r\in\brow^*(t)$.
The first term is determined by the children, and we have already
ensured that it can be made arbitrarily small by choosing the
childrens' ranks sufficiently large.

This leaves us only with the second term.
We follow the same approach as before: we combine all admissible
blocks in a large matrix
\begin{align*}
  \widehat{G}_t &:= U_t^* G_t
  = \begin{pmatrix}
       U_t^* G_{tr_1}^c & \ldots & U_t^* G_{tr_m}^c
    \end{pmatrix}, &
  \brow(t) &= \{r_1,\ldots,r_m\}
\end{align*}
and compute its singular value decomposition.
The first $k$ left singular vectors form the columns of $\widehat{V}_t$,
and splitting $\widehat{V}_t$ according to (\ref{eq:Ut_Vhatt}) provides
us with the transfer matrices $E_{t_1}$ and $E_{t_2}$.

In order to set up the matrix $\widehat{G}_t$, we require the matrices
$U_t^* G_{tr}^c$ for all $r\in\brow^*(t)$, and computing these
products explicitly would be too time-consuming.
By definition, we have $r\in\brow^*(t_1)$ and $r\in\brow^*(t_2)$
and therefore
\begin{equation*}
  U_t^* G_{tr}^c
  = \begin{pmatrix}
       V_{t_1}^* G_{t_1 r}^c\\
       V_{t_2}^* G_{t_2 r}^c
    \end{pmatrix}.
\end{equation*}
To take advantage of this equation, we introduce auxiliary
matrices $\widehat{G}_{tr}^c := V_t^* G_{tr}^c$ for all
$t\in\ctI$, $r\in\brow^*(t)$.
On one hand, with these matrices at our disposal, we can quickly
set up $\widehat{G}_t$ by copying appropriate submatrices.

On the other hand, we can prepare the matrices $\widehat{G}_{tr}^c$
efficiently:
for leaf clusters $t$, $G_{tr}^c$ has only a small number of rows,
allowing us to compute $\widehat{G}_{tr}^c$ directly.
For non-leaf clusters, we have
\begin{equation*}
  \widehat{G}_{tr}^c = V_t^* G_{tr}^c
  = \widehat{V}_t^* U_t^* G_{tr}^c
  = \widehat{V}_t^* \begin{pmatrix}
      V_{t_1}^* G_{t_1 r}^c\\
      V_{t_2}^* G_{t_2 r}^c
    \end{pmatrix}
  = \widehat{V}_t^* \begin{pmatrix}
      \widehat{G}_{t_1 r}^c\\
      \widehat{G}_{t_2 r}^c
    \end{pmatrix}
\end{equation*}
and can compute $\widehat{G}_{tr}^c$ in $\mathcal{O}(k^3)$ operations.
With this modification, we can construct a row basis with guaranteed
accuracy in $\mathcal{O}(n k^2 \log n)$ operations.
This same holds for the column basis.

Once we have the row and the column basis at our disposal, we can
again take advantage of the factorized form
$G|_{\hat t\times\hat r} = A_{tr} B_{tr}^*$ to compute the corresponding
coupling matrices
\begin{align*}
  V_t^* G|_{\hat t\times\hat r} W_r
  &= V_t^* A_{tr} (W_r^* B_{tr})^* &
  &\text{ for all } b=(t,r)\in\lfaIK
\end{align*}
with $\mathcal{O}((|\hat t|+|\hat r|)k^2)$ operations per admissible
block for a total complexity of $\mathcal{O}(n k^2 \log n)$.

The entire $\mathcal{H}^2$-matrix compression requires
$\mathcal{O}(n k^2 \log n)$ operations under standard assumptions
\cite[Section~6.5]{BO10}.

%
%
\subsection{Error control}

We have already seen that we can adjust the coarsening procedure
to ensure block-relative error control.
We can introduce a small modification to the $\mathcal{H}^2$-matrix
recompression algorithm that provides similar error control
following \cite[Section~6.8]{BO10}.

As before, we focus only on the row basis and briefly remark that
the column basis can be treated in exactly the same way by applying
the algorithm to the adjoint matrix $G^*$ with the adjoint blocks
$(s,t)$ instead of $(t,s)\in\lfaIJ$.

In order to ensure a block-relative error bound for an admissible
block $b=(t,r)\in\lfaIK$, we recall that it has a low-rank representation
$G|_{\hat t\times\hat r} = A_{tr} B_{tr}^*$.
We have already seen that we can apply an orthogonal transformation
from the right to obtain a condensed matrix $G_{tr}^c$ with only $k$
columns such that $\|G|_{\hat t\times\hat r}\|_2 = \|G_{tr}^c\|_2$, and
we can use a few steps of the power iteration to obtain a lower bound
for the spectral norm.

We are looking for block-relative error control, i.e., we would
like to ensure
\begin{equation*}
  \|G|_{\hat t\times\hat r} - V_t V_t^* G|_{\hat t\times\hat r}\|_2
  \leq \hat\epsilon \|G|_{\hat t\times\hat r}
  \qquad\iff\qquad
  \|G_{tr}^c - V_t V_t G_{tr}^c\|_2
  \leq \hat\epsilon \|G_{tr}^c\|_2
\end{equation*}
for a given $\hat\epsilon\in\bbbr_{>0}$.
For leaf clusters $t\in\lfI$, we have $V_t$ are our disposal and can
use it to control the error directly.
For non-leaf clusters $t\in\ctI\setminus\lfI$, we recall
(\ref{eq:error_decomp}) to obtain
\begin{align*}
  \|G_{tr}^c - V_t V_t^* G_{tr}^c\|_2^2
  &= \|G_{tr}^c - U_t U_t^* G_{tr}^c\|_2^2
   + \|\widehat{G}_{tr}^c - \widehat{V}_t \widehat{V}_t^*
       \widehat{G}_{tr}^c\|_2^2\\
  &\leq \|G_{t_1r}^c - V_{t_1} V_{t_1}^* G_{t_1r}^c\|_2^2
   + \|G_{t_2r}^c - V_{t_2} V_{t_2}^* G_{t_2r}^c\|_2^2
   + \|\widehat{G}_{tr}^c - \widehat{V}_t \widehat{V}_t^*
       \widehat{G}_{tr}^c\|_2^2,
\end{align*}
i.e., the error is bounded by an orthogonal sum of the errors in the children
and the error added by the parent.
A simple induction yields
\begin{equation*}
  \|G_{tr}^c - V_t V_t^* G_{tr}^c\|_2^2
  \leq \sum_{s\in\desc(t)} \|\widehat{G}_{sr}^c
           - \widehat{V}_s \widehat{V}_s^* \widehat{G}_{sr}^c\|_2^2,
\end{equation*}
where we have added $\widehat{G}_{sr}^c = G_{sr}^c$ and
$\widehat{V}_s = V_s$ for leaf clusters $s\in\lfI$ for consistency.

In order to keep the relative error under control, we introduce
weight factors
\begin{align*}
  \omega_{sr} &:= \theta^{(\level(s)-\level(t))/2} \|G_{tr}^c\|_2 &
  &\text{ for all } s\in\desc(t),\ (t,r)\in\lfaIJ
\end{align*}
and scale the submatrices $\widehat{G}_{sr}^c$ by $\omega_{sr}^{-1}$
so that the singular value decomposition now guarantees
\begin{align*}
  \|\widehat{G}^c_{sr} - \widehat{V}_s \widehat{V}_s^*
    \widehat{G}^c_{sr}\|_2^2 &\leq \epsilon^2 \omega_{sr}^2 &
  \text{ for all } s\in\desc(t),\ (t,r)\in\lfaIJ.
\end{align*}
If we assume that $|\chil(s)|\leq\sigma$ holds for all $s\in\desc(t)$,
a simple induction yields that there are no more than
$\sigma^{\ell-\level(t)}$ descendants of $t$ on level $\ell$, and we can
use the geometric sum to find
\begin{align*}
  \|G_{tr}^c - V_t V_t^* G_{tr}^c\|_2^2
  &\leq \sum_{s\in\desc(t)} \epsilon^2 \theta^{\level(s)-\level(t)}
           \|G_{tr}^c\|_2^2\\
  &\leq \epsilon^2 \|G_{tr}^c\|_2^2
           \sum_{\ell=\level(t)}^\infty (\sigma\theta)^{\ell-\level(t)}
   = \epsilon^2 \|G_{tr}^c\|_2^2 \frac{1}{1 - \sigma\theta}.
\end{align*}
If we want to ensure a block-relative error below $\hat\epsilon$,
we only have to choose $\theta < 1/\sigma$ and
$\epsilon := \hat\epsilon \sqrt{1-\sigma\theta}$.
We conclude that it is enough to introduce scaling factors to the
matrices $\widehat{G}_{tr}^c$ in the basis construction to ensure
block-relative error control in every admissible block
simultaneously.

\section{Numerical experiments}

The new algorithm for the approximation of the product of two
$\mathcal{H}^2$-matrices starts by computing an \emph{exact}
intermediate representation of the product, then applies simple
low-rank truncations to coarsen the resulting block structure,
and then converts the intermediate approximation into the
final $\mathcal{H}^2$-matrix.

Since both the coarsening algorithm and the final conversion
are based on singular value decompositions, we can perfectly
control the resulting error.
This is a major advantage compared to standard $\mathcal{H}$-matrix
methods that have to apply compression already to intermediate
results and therefore cannot guarantee blockwise relative accuracies.

To demonstrate the performance of the new algorithm, we set up
classical boundary integral operators for the Laplace equation:
the single-layer matrix
\begin{align*}
  g_{ij} &= \int_{\partial\Omega} \varphi_i(x)
           \int_{\partial\Omega} g(x,y) \varphi_j(y) \,dy \,dx &
  &\text{ for all } i,j\in\{1,\ldots,n\}\\
\intertext{and the double-layer matrix}
  k_{ij} &= \int_{\partial\Omega} \varphi_i(x)
           \int_{\partial\Omega} \frac{\partial g}{\partial n(y)} \varphi_j(y)
                   \,dy \,dx &
  &\text{ for all } i,j\in\{1,\ldots,n\}
\end{align*}
with piecewise constant basis functions $(\varphi_i)_{i=1}^n$ and the
Laplace kernel function
\begin{align*}
  g(x,y) &= \begin{cases}
              \frac{1}{4\pi \|x-y\|} &\text{ if } x\neq y,\\
              0 &\text{ otherwise}
            \end{cases} &
  &\text{ for all } x,y\in\bbbr^3.
\end{align*}
For the single-layer matrix $G$, we use the surface of the unit
sphere $\Omega=\{x\in\bbbr^3\ :\ \|x\|_2\leq 1\}$, equipped with a
triangular mesh.

Since the double-layer operator on the unit sphere is very similar
to the single-layer operator, we use the surface of the \emph{unit cube}
$\Omega=[-1,1]^3$ for the double-layer matrix, again with a triangular
mesh.

We use hybrid cross approximation \cite{BOGR04} to construct a preliminary
$\mathcal{H}$-matrix approximation of the matrices $G$ and $K$ and then
use the algorithm of \cite[Section~6.5]{BO10} to obtain the
$\mathcal{H}^2$-matrix required for the new algorithm.
Then the new algorithm is used to approximate the products $G^2$ and $K^2$,
respectively.

%
%
\begin{table}
  \begin{equation*}
    \begin{array}{r|rrrr}
      n & \text{Row}/s & \text{Col}/s & \text{Mem/MB} & \text{rel. error}\\
      \hline
     2\,048 & 1.4 & 0.7 & 4.7 & 2.6_{-5}\\
     4\,608 & 4.9 & 2.8 & 10.6 & 2.8_{-5}\\
     8\,192 & 10.8 & 6.1 & 20.7 & 3.0_{-5}\\
    18\,432 & 32.7 & 19.3 & 42.1 & 3.5_{-5}\\
    32\,768 & 68.2 & 40.6 & 77.5 & 4.0_{-5}\\
    73\,728 & 187.9 & 113.5 & 162.4 & 3.8_{-5}\\
   131\,072 & 379.2 & 232.0 & 298.1 & 4.0_{-5}\\
   294\,912 & 986.5 & 607.4 & 628.8 & 4.0_{-5}\\
   524\,288 & 1930.6 & 1199.5 & 1191.8 & 4.0_{-5}\\
1\,179\,648 & 4945.0 & 3081.6 & 2474.6 & 4.1_{-5}\\
2\,097\,152 & 9446.3 & 5785.2 & 4662.1 & 4.1_{-5}
    \end{array}
  \end{equation*}
  \caption{Multiplication of $\mathcal{H}^2$-matrices for the
    single-layer operator}
  \label{ta:slp}
\end{table}

Table~\ref{ta:slp} shows the results of a first series of experiments,
carried out on one core of an AMD EPYC 7713 processor using the
reference BLAS and LAPACK libraries for the single-layer matrix.
The first column gives the matrix dimension, ranging from $2\,048$ to
$2\,097\,152$.
The second column gives the time in seconds required to construct the
intermediate basis-tree representation of the product and to choose the
row cluster basis.
The third column gives the time in seconds requires for the column cluster
basis.
Our implementation re-uses the basis-tree representation that was
provided for the construction of the row basis, thereby saving time.
The fourth column gives the storage requirements in MB for the
$\mathcal{H}^2$-matrix product.
The fifth column lists the relative spectral error of the result,
estimated by ten steps of the power iteration.

%
%
\begin{table}
  \begin{equation*}
    \begin{array}{r|rrrr}
      n & \text{Row}/s & \text{Col}/s & \text{Mem/MB} & \text{rel. error}\\
      \hline
     3\,072 & 6.0 & 2.7 & 16.6 & 7.3_{-6}\\
     6\,912 & 24.4 & 14.3 & 38.4 & 7.4_{-6}\\
    12\,288 & 50.3 & 28.4 & 67.1 & 8.1_{-6}\\
    27\,648 & 175.6 & 110.7 & 151.9 & 8.0_{-6}\\
    49\,152 & 332.8 & 203.5 & 258.6 & 8.4_{-6}\\
   110\,592 & 1056.3 & 680.1 & 587.3 & 9.1_{-6}\\
   196\,608 & 1903.4 & 1192.9 & 1008.7 & 9.7_{-6}\\
   442\,368 & 6211.1 & 4259.7 & 2377.2 & 1.0_{-5}\\
   786\,432 & 10138.6 & 64864.1 & 4130.9 & 1.1_{-5}\\
1\,769\,472 & 29551.4 & 19127.9 & 9928.6 & 1.1_{-5}\\
3\,145\,728 & 49076.4 & 32054.4 & 17353.2 & 1.3_{-5}
    \end{array}
  \end{equation*}
  \caption{Multiplication of $\mathcal{H}^2$-matrices for the
    double-layer operator}
  \label{ta:dlp}
\end{table}

Table~\ref{ta:dlp} contains similar results for the double-layer matrix.

We have prescribed an error bound of $10^{-4}$ for this experiment, and
seeing that the error obtained by our algorithm is always below this
bound allows us to conclude that the error control strategy works.

%
%
\begin{figure}
  \pgfdeclareimage[width=12cm]{h2arith}{h2arith_mul}

  \begin{pgfpicture}{0cm}{0cm}{12cm}{7.5cm}
    \pgfuseimage{h2arith}
  \end{pgfpicture}
  \caption{Runtime per degree of freedom for the matrix multiplication}
  \label{fi:runtime}
\end{figure}

Figure~\ref{fi:runtime} illustrates that our experiments confirm
the theoretical bound of $\mathcal{O}(n k^2 \log n)$ for the run-time
of the new algorithm:
if we divide the run-time by the matrix dimension $n$ and plot the
result using a logarithmic scale for $n$, we can see that the
time per degree of freedom appears to grow like $\log n$, just as
predicted by our theory.

In conclusion, we have found an algorithm that approximates the
product of two $\mathcal{H}^2$-matrices by a new $\mathcal{H}^2$-matrix
using adaptively chosen cluster bases.
The algorithm constructs the \emph{exact} product in an intermediate
step, allowing us to guarantee, e.g., \emph{block-relative} error
bounds for the result.
Using the specially-designed basis-tree representations, the entire
computation can be performed in $\mathcal{O}(n k^2 \log n)$
operations, and both this efficiency and the validity of the error
control strategy are confirmed by our experiments.

The application of this new algorithm to more challenging applications
like the construction of matrix inverses and triangular factorizations
is the subject of ongoing research.

\bibliography{hmatrix}
\bibliographystyle{plain}

\end{document}